\documentclass[12pt]{amsart}
\usepackage{amssymb,latexsym}
\usepackage{comment}
\usepackage{enumitem} 
\newtheorem{theorem}{Theorem}[section]
\newtheorem{prop}[theorem]{Proposition}
\newtheorem{lemma}[theorem]{Lemma}
\newtheorem{remark}[theorem]{Remark}

\newtheorem{definition}[theorem]{Definition}

\newcommand{\cC}{\mathcal{C}}
\newcommand{\V}{\mathcal{V}}
\newcommand{\E}{\mathcal{E}}

\newcommand{\R}{\mathbb{R}}

\newcommand{\ep}{\epsilon}
\newcommand{\PP}{\mathbb{P}}

\begin{document}

\title{Symplectic blowing down in dimension Six}

\author{Tian-Jun Li \& Yongbin Ruan \& Weiyi Zhang}
\address{School  of Mathematics\\  University of Minnesota\\ Minneapolis, MN 55455}
\email{tjli@math.umn.edu}
\address{Institute for Advanced Study in Mathematics\\East No.7 Building, Zhejiang University Zijingang Campus\\ Hangzhou, China}
\email{ruanyb@zju.edu.cn}
\address{Mathematics Institute\\ University of Warwick\\Coventry,
CV4 7AL, UK}\email{weiyi.zhang@warwick.ac.uk}

\maketitle

\begin{abstract} We establish a blowing down criterion in the context of birational symplectic geometry in dimension 6.
\end{abstract}

\tableofcontents
\section{Introduction}\label{contraction}

Symplectic blowing down is a   fundamental operation in birational symplectic geometry
(\cite{GS}, \cite{HLR}).
Two symplectic manifolds are birationally cobordant equivalent if 
one can be obtained from the other via a sequence of
symplectic blow ups, symplectic blow downs and integral deformations. 

Symplectic blowing up can always be performed along any symplectic submanifold  to get a new symplectic manifold with an exceptional  divisor (\cite{Gr}, \cite{GS}, \cite{McS}). On the other hand, except in dimension 4 (\cite{Mc}), 
it is poorly understood when blowing down  can be performed. In this note
we will attempt to  find     blowing down criteria in dimension 6. 
Blowing up in dimension 6 gives rise to either a
symplectic $\mathbb P^2$ with normal degree $-1$ or a symplectic
$\mathbb P^1-$bundle with normal degree $-1$ along the $\mathbb P^1-$fibers.  
The geometry of these  symplectic 4-manifolds are well understood (\cite{Mc}, \cite{LL1}, \cite{OO}, \cite{LM}). 
Our focus is whether such a symplectic  divisor always arises from a
symplectic blowing up. 

The case of $\mathbb P^2$ is simpler. We observe that  the uniqueness of  the symplectic structures by Gromov and Taubes, 
together with  the Weinstein neighborhood theorem, implies that  a
symplectic $\mathbb P^2$ divisor with normal degree $-1$ in a symplectic $6-$manifold can always be blown down
just as in the case of $\mathbb P^1$ with self-intersection $-1$ in
a symplectic $4$--manifold.

 The case of a $\mathbb P^1-$bundle over a
Riemann surface $\Sigma$ is   subtler. 
Topologically,  blowing down  can always be performed since it is the same as topologically  fiber summing  with the  
triple of a linear $\mathbb P^2-$bundle with  a $\mathbb P^1-$subbundle
over $\Sigma$ with opposite normal bundle and a complementary section.
However, unlike the $\mathbb P^2$ case, it is not clear that blowing down can always be performed symplectically in this case.
Symplectic blowing up generally involves a small neighborhood of $\Sigma$ which means that the resulting $\mathbb P^1-$bundle divisor has a small fiber area. 

We investigate this problem in the context of birational symplectic geometry in dimension 6 (which has been studied in  \cite{R1}, \cite{R2}, \cite{Hu}, \cite{tLR}, \cite{LR}, \cite{Vo}, \cite{Ti} etc). Specifically, we study  whether we can symplectically blow down 
up to an integral deformation to obtain a simpler birationally equivalent symplectic manifold. 
A nice feature is that cohomologous symplectic forms on such $4-$manifolds are also isotopic in this case. 
 Moreover, as every symplectic structure on such a
4--manifold is K\"ahler, we can apply algebro-geometric techniques
to solve this problem.

In Section 2 we investigate this problem in a general situation. 
Topologically the blowing up construction gives rise to a fibred codimension 2 submanifold as in the following definition.

\begin{definition}\label{smooth blowup}
Let  $D^{2n}\subset (M^{2n+2}, \omega)$  be a codimension 2 symplectic submanifold which admits  
 a linear $\mathbb P^{k}-$bundle structure $\pi:D^{2n}\to Y^{2n-2k}$ 
over an oriented $(2n-2k)-$manifold $Y$. $(D, \pi)$ is called a topological  exceptional divisor if (i) the  normal line bundle  $N_D$ is the tautological line bundle when restricted to the projective space fibers of $\pi$, and (ii) $\omega|_D$ is almost standard. 
\end{definition}

Here a linear $\PP^{k}-$bundle refers to  a projective bundle from a complex vector bundle of dimension $k+1$. So the structure group
is the linear group $GL(k+1, \mathbb C)$. And the fibers of such a bundle 
come with a homotopy class of almost complex structures and hence 
have a natural complex orientation and a line class $l$ in $H_2$. We use $l$ to specify the forward cone 
$$\cC_{\pi}(D)=\{u\in H^2(D, \R)| u^n>0, \langle u, l\rangle >0\}.$$
A fibred symplectic form on  a linear projective bundle  is called standard if it arises from the Sternberg-Weinstein universal construction as described in Section 2.1.
In particular, a standard form restrict to a multiple of the Fubini-Study form on each fiber. 
A fibred symplectic form is said to be almost standard  if it is 
 deformation 
 to a standard form via fibred forms.

We describe  the blowing up construction  in  several ways, via the $U(k)$ universal construction, birational cobordism and symplectic cut. From Definition \ref{blowup} we see that   symplectic exceptional divisor from blowing up   is a  topological  exceptional divisor.

The  symplectic cut description is useful to  study whether a topological exceptional divisor can be blown down up to deformation. 
Specifically it  introduces an auxiliary  $S^1-$equivariant linear symplectic projective space bundle triple (see Lemma \ref{triple}), which lies behind the following definition.

\begin{definition}\label{matching triple}
Given a topological  exceptional divisor $\pi: D^{2n}\to Y^{2n-2k}$ of $(M^{2n+2}, \omega)$, 
we define a matching  triple $(X, D', S; \Omega)$ to be a  linear  $(\mathbb P^{k+1}, \mathbb P^{k}, \mathbb P^0)$ bundle triple $(X, D', S)$ over $Y$ with $\Omega$  a symplectic form on $X$, satisfying the following conditions. 
\begin{enumerate}

\item $D'$ is diffeomorphic to $D$;

\item $e(N_{D'})=-e(N_{D})$;

\item  $\Omega|_{D'}$ is almost standard and matches with $\omega|_{D}$;

\item $\Omega$ is $S^1-$invariant with respect to a semi-free $S^1-$action with $D'$ and $S$ as the only fixed point sets.
\end{enumerate}
The triple is called weak  if  we assume the weaker condition 

\begin{center}
(3') $\Omega|_{D'}$ is almost standard and  $[\Omega|_{D'}]=[\omega|_{D}]$.
\end{center}
\end{definition}

Here is the main result in Section 2. 

\begin{prop} \label{fiber sum} Suppose $(M, \omega)$ has a topological  exceptional divisor $D$. 
Up to integral deformation, we can symplectically blow down $(M, \omega)$ along $D$ if  there is  a matching  triple $(X, D', S; \Omega)$. 
\end{prop}
Thus we are led to search for such projective space bundle triples which match with $(M, D, \omega)$. 
 For  blowing up a surface in arbitrary dimension,
the auxiliary linear symplectic projective space bundle triple also  leads to a  ratio constraint.

\begin{theorem} \label{main} 
Suppose a symplectic divisor $\pi:D^{2n}\to \Sigma$ of $ (M^{2n+2}, \omega)$ arises from  symplectically    blowing up a symplectic surface $\Sigma$ 
in a symplectic manifold. Then

\begin{enumerate}[label=(\roman*)]
\item   
$c_1(N_D)$ satisfies
$$ \int_D (-1)^n c_1(N_D)^n=-\deg(N_\Sigma),$$

\item  the symplectic form $ \omega|_D$   satisfies the ratio bound
\begin{equation}\label {constraint}  
    \rho_{\pi} ([\omega|_D]) >
     \left\{ \begin{array}{ll} -\deg(N_{\Sigma}),  & \hbox{if $g(\Sigma)>0$}, \\
\max\{-\deg(N_{\Sigma}), \quad t_n(-\deg (N_{\Sigma}))\},       &\hbox{if $g(\Sigma)=0$}.\\
\end{array}
\right.
\end{equation} 
\end{enumerate}
\end{theorem}

What is  essentially new is the ratio inequality \eqref{constraint}. We now explain the terms in this theorem. 
Here $t_n$ is the    function $$t_n:\mathbb Z\to \{0, 1,\cdots, n-1\}$$ by requiring $t_n(w)= w\pmod n.$ 
The degree function $\deg(E)$ of  a  symplectic vector bundle $E$ over $\Sigma$  is $\deg(E)=\int_\Sigma c_1(E)$. 
The ratio function $\rho_{\pi}$ on the forward cone $\cC_{\pi}(D)$ for a linear projective space bundle $\pi: D^{2n}\to \Sigma$ is defined by
\begin{equation}\label{ratio of a class} \rho_{\pi}(u)=\frac{\int_D u^n}{ \langle u, l\rangle ^n}\in \mathbb R^+.\end{equation}   
Note that the ratio $\rho_{\pi}(u)\in \mathbb R^+$   determines the ray of the class $u$ since $\rho_{\pi}(u)$  is scale invariant  and the rank of $H^2(D;\mathbb R)$ is $2$ in this case.

We introduce the following definition.

\begin{definition}\label{con blowup}
A topological  exceptional divisor $\pi:D\to \Sigma$ of $(M, \omega)$  is called  admissible  if  the  symplectic form $\omega|_D$ on $D$ 
satisfies the ratio bound  
\begin{equation} \label {ratio constraint} 
    \rho_{\pi} ([\omega|_D]) >
     \left\{ \begin{array}{ll}\alpha_{D, M},  & \hbox{if $g(\Sigma)>0$}, \\
\max\{\alpha_{D, M}, \quad t_n(\alpha_{D, M})\},      &\hbox{if $g(\Sigma)=0$}.\\
\end{array}
\right.
\end{equation} 
where $\alpha_{D, M}=\int_D (-1)^n c_1(N_D)^n$.
\end{definition}

A  symplectic exceptional divisor from blowing up a surface  is an admissible topological  exceptional divisor by Theorem \ref{main}.

\begin{theorem} \label{coh  matching}
Suppose $\pi:D\to \Sigma$  is an admissible topological  exceptional divisor of $(M, \omega)$.  Then there exists a weak matching $S^1-$equivariant triple. 
\end{theorem}

 To prove this result we need to understand the
symplectic cone of  a linear  $(\mathbb P^{k+1}, \mathbb P^{k}, \mathbb P^0)$ bundle triple over $\Sigma$. This is possible in the K\"ahler setting by the 
Kleiman's criterion.
  We are able to determine
the restricted K\"ahler cone for various complex structures coming
from holomorphic bundles of the form $\mathcal V\oplus \mathcal O$ over $\Sigma$ with $\mathcal V$ either semi-stable or decomposable.  
Another  nice feature in the K\"ahler setting is that we always have the required $S^1-$symmetry.

Along the way we determine the symplectic cone of an arbitrary linear projective bundle over a surface. 
\begin{prop} \label{almost standard cone}
For  a linear $\mathbb P^{n}$ bundle over a surface, 
\begin{enumerate}[label=(\roman*)]
\item  
 every almost standard symplectic form is cohomologous to a K\"ahler form,

\item  the symplectic cone is equal to the positive cone   when $n$ is odd and the surface is of positive genus. 
\end{enumerate}
\end{prop}

\begin{remark}
This proposition generalizes the result of \cite{Mc0} in  the case of  $\mathbb P^1-$bundles.
Cascini-Panov \cite{CP} noted that   the generic K\"ahler cone is smaller than the  symplectic cone  for one point blowup when the genus is 1. 
\end{remark}

Finally, we state our blowing-down criterion in dimension 6. 
In dimension 6 a symplectic exceptional divisor arising from blowing up is a topological exceptional divisor $\pi:D\to Y$
as in Definition \ref{smooth blowup}, where $Y$ is either a point or a surface $\Sigma$. 

When $Y$ is a point,     $D$ is a symplectic $\mathbb P^2$
embedded in $(M, \omega)$ whose  normal bundle $N_D$ has degree $-1$.
As already mentioned, 
since cohomologous  symplectic forms on $\mathbb P^2$ are diffeomorphic (\cite{T}),
$\omega|_D$ is standard and 
a neighborhood of $(D,\omega)$
is the same as the standard symplectic ball near the boundary.
So it can be symplectically blown down.
When $Y=\Sigma$, $D$ is a ruled surface.
Since cohomologous symplectic forms are also diffeomorphic for ruled surfaces (\cite{LM}), 
we have  by Proposition \ref{fiber sum} and Theorem \ref{coh matching}, 

\begin{prop} \label{dim 6} Let $(M, \omega)$ be a $6-$dimensional symplectic manifold
and $\pi:D^4\to Y$ a topological exceptional divisor of $(M, \omega)$. 
When $Y$ is a point, $(M, \omega)$  can be  
blown down along $D^4$ to $(M', \omega')$. 

When $Y=\Sigma$ and    $\pi: D^4\to \Sigma$  is admissible then  a  weak matching triple is a matching triple and hence, 
     up to integral deformation,
  $(M, \omega)$   
 can be 
 blown down along $D^4$ to $(M', \omega')$. 
 
 \end{prop}

 It would be interesting to investigate, when $Y=\Sigma$,  whether integral deformation is not really needed. 
However, note that the integral deformation is a birational equivalence so  the symplectic manifolds $(M, \omega)$ and $(M', \omega')$ are birational. 
So this blowing-down criterion in dimension $6$ is suitable in the context of symplectic birational geometry.

 In the case of $Y=\Sigma$,  since any $\mathbb P^1-$bundle over $\Sigma$  is linear and any symplectic form is fibred (Section 6.2 in \cite {McS}),  we in fact have the following more explicit and stronger formulation.

 \begin{theorem} \label{ruled} Let $(M, \omega)$ be a $6-$dimensional symplectic manifold and $D$ a codimension $2$ symplectic submanifold.
 Suppose $D$ admits a  $\mathbb P^1-$bundle structure $\pi:D\to \Sigma$   over a surface $\Sigma$ with $\langle c_1(N_D), l\rangle =-1$. 
Let $\alpha_{D, M}=c_1(N_D)\cdot c_1(N_D)$ and $\rho=\rho_{\pi}([\omega|_D])$.  

Suppose  $(D, \omega|_D)$ arises from   blowing up a surface.  Then $\rho \ne 2$ if   $D= S^2\times S^2$ with $\alpha_{D, M}=2$,  
and $\rho>\alpha_{D, M}$  otherwise.

Conversely,  if  $\rho>\alpha_{D, M}$,   
$(M, \omega)$ can be  
blown down  along $D$ up to deformation.   In particular, this is  the case if $\alpha_{D, M}\leq  0$.
Moreover, when  $D=S^2\times S^2$ with $\alpha_{D, M}=2$, $(M, \omega)$ can be 
blown down  along $D$ up to deformation as long as $\rho\ne 2$.
 \end{theorem}

When $D=S^2\times S^2$ and $\alpha_{D, M}=2$,   there are two rulings and 
the normal bundle $N_D$ has degree $-1$ along each ruling. 
The condition $\rho\ne 2$, just means  the symplectic areas of
the two rulings are not the same.  In this case,   up to deformation, $(M, \omega)$  can be blown down 
 along
the ruling  with smaller area. 
 This picture is consistent with
the flop operation for projective 3--folds. Let $D=S^2\times S^2$ be a divisor with normal $c_1=(-1, -1)$. If the fibers of the two rulings $S^2\times \{pt\}$ and $\{pt\} \times S^2$ have the same symplectic area and are not cohomologous, we can perturb the symplectic form such that the areas of two $S^2$ factors are different. A generic perturbation would work as the symplectic cone of the ambient manifold is open and we can perturb along any direction. Hence, in this situation, the divisor $D$ could also be blown down up to deformation. 
Simplest such example might be the toric blowup of the projective cone over $S^2 \times  S^2$ at the conic point.

In Section 2, we   review the blow up  process, establish the criterion Proposition \ref{fiber sum}.
In Section 3 we study the curve cone  and the K\"ahler cone of holomorphic projective bundles over a Riemann surface, and also prove Proposition \ref{almost standard cone}.
In Section 4 we   prove bound \eqref{constraint}, Theorem \ref{coh matching} and   Theorem \ref{ruled}. 

We thank Bob Gompf and Jianxun Hu for useful conversations. The first and second authors are grateful to the support of NSF, and the third author is grateful to the support of EPSRC during the preparation of the manuscript. 
We dedicate this paper to Professor Banghe Li on the occasion of his 80th birthday.

\section{Blowing up/down, symplectic cut and fiber sum}


\subsection{Blowing up via the $U(k)$ universal construction}

We first recall the Sternberg-Weinstein universal construction to obtain a canonical symplectic structures on the normal bundle 
of a symplectic submanifold. 
\subsubsection{The universal construction}
Let us first review the Sternberg-Weinstein universal construction.
Let $\pi:P\to X$ be a principal bundle with structure group $G$ over
a symplectic manifold $(X, \omega)$. If ${\mathfrak
g}$ denote the Lie algebra of $G$ and ${\mathfrak g}^*$ denotes the dual, the Sternberg-Weinstein universal construction produces a
$G-$invariant 
symplectic form on a neighborhood of $P\times \{0\}$ in $P\times {\mathfrak g}^*$.

The construction starts with   a connection on $P$, which is  a $G-$invariant $\mathfrak g-$valued $1-$form $A$ on $P$ corresponding to a $G-$invariant 
projection onto the vertical tangent bundle $VP$.  Equivalently, it is given by a  $G-$invariant complementary subbundle and, dually,   it induces an embedding of  $P\times {\mathfrak g}^*$ into
$T^*P$.
 Consider the  $1-$form on  $P\times {\mathfrak g}^*$ given by $\gamma\cdot A$ at $(p, \gamma)\in P\times {\mathfrak g}^*$, where
we use $\cdot$ to denote the pairing between $\mathfrak g$ and $\mathfrak g^*$.
Denote this $1-$form by $\gamma\cdot A$ as well. Notice that $d(\gamma\cdot A)$ is the
restriction of the canonical $2-$form on $T^*P$. Therefore
$d(\gamma\cdot A)$ is non-degenerate on the fibers of $P\times {\mathfrak
g}^*$.
The $2-$form  $$\omega_A=\pi^*\omega+d(\gamma\cdot A)$$ is 
called the coupling form of $A$. The $G-$action on $P\times \mathfrak g^*$
given by
$$g(p, \zeta)=(g^{-1}p, \hbox{Ad}(g)^*\zeta),$$
preserves $\gamma\cdot A$ and hence $\omega_A$.
Notice that at any point in $P\times \{0\}$, $\gamma=0$ and
$\omega_A$ is equal to $\pi^*\omega+d\gamma \wedge A$, hence it is
symplectic there.

\begin{lemma} \label{W} The form $\omega_A$ is
a symplectic form on $P\times {\mathcal W_A}$ for some $G-$invariant
neighborhood ${\mathcal W_A}$ of $0\in \mathfrak g^*$. The projection onto
${\mathcal W_A}$ is a moment map  on $P\times {\mathcal W_A}$.
\end{lemma}

This lemma is well-known. Notice that the vertical bundle $VP$ is the
bundle of the null vectors of $\pi^*\omega$. So as explained in
\cite{GS}, the construction is just a special case of the
coisotropic embedding theorem. It follows from the uniqueness part
of the coisotropic embedding theorem that the symplectic structure
$\omega_A$ on $P\times {\mathcal W}_A$ near $P\times 0$ is
independent of $A$ up to symplectomorphisms.

More generally, if $(F, \omega_F)$ is a symplectic manifold with a
Hamiltonian $G$ action, we can form the associated   bundle
$P_F=P\times_G F$. Let $\mu_F:F\to \mathfrak g^*$ be a moment map.
Furthermore, assume that
\begin{equation} \label{W2}
\mu_F(F)\subset \mathcal W_A.
\end{equation}
Then there is a symplectic structure $\omega_{F, A}$ on $P_F$ which
restricts to $\omega_F$ on each fiber.

To construct  $\omega_{F, A}$ consider the $2-$form
$\omega_A+\omega_F$ on $P\times \mathfrak g^*\times F$. It is invariant
under the diagonal $G-$action and is symplectic on $P\times
{\mathcal W_A}\times F$.
The $G-$action is Hamiltonian with
$$\Gamma_{\mathcal W_A}=\pi_{\mathfrak g^*}+\mu_F:P\times {\mathcal W_A}\times F\to \mathfrak g^*$$
as a moment map.
Furthermore, by (\ref{W2}), for any $f\in F$, we have $\mu_F(f)\in
{\mathcal W_A}$, thus
$$\Gamma_{\mathcal W_A}^{-1}(0)=\{(p, -\mu_F(f), f)\}.$$
In particular, $\Gamma_{\mathcal W_A}^{-1}(0)$ is $G-$equivariantly
diffeomorphic to $P\times F$, and the symplectic reduction at $0$
yields the desired symplectic form $\omega_{F, A}$ on $P_F$.

In fact when $(F, \omega_F)=(TG^*, \omega_{can})$,
 then  $P\times \mathfrak g^*=P\times_G TG^*$.

\subsubsection{The local geometry of a submanifold}
We apply this construction to obtain canonical symplectic structures on small normal disk bundles 
of a symplectic submanifold. 
Suppose that $X$ is a closed symplectic manifold of dimension $2n$
and $Y\subset X$ is a symplectic submanifold of codimension $2k$.
The normal bundle $N_Y$ is a symplectic vector bundle, {\it i.e.} a bundle
with fiber $(\mathbb C^k, \tau)$. Here $\tau$ denotes the standard symplectic form on $\mathbb C^k$. Note that $\tau$ is $U(k)$ invariant. 
Picking a compatible
almost complex structure on $N_Y$, we then have an Hermitian bundle.
Let $P$ be the principal $U(k)$ bundle over $Y$.

Now pick a unitary connection $A$ for $P$, and let $\mathcal
W_A\subset u(k)^*$ be as in Lemma \ref{W}. Let
$D_{\epsilon_0}\subset \mathbb C^k$ be the closed
$\epsilon_0-$ball such that its image under  the moment map lies
inside $\mathcal W_A$.
Apply the universal construction to $P$ and  $D_{\epsilon_0}\subset
\mathbb C^k$, we get 
a $U(k)-$invariant  symplectic  form $\omega_{\epsilon_0, A}$ on
the disc bundle $N_Y(\epsilon_0)$ which restricts to $\tau$
on each fiber and restricts to $\omega|_Y$ on the zero section.

 By the symplectic neighborhood
theorem, and by possibly taking a smaller $\epsilon_0$, a tubular
neighborhood ${\mathcal N}_{\epsilon_0}(Y)$ of $Y$ in $X$ is
symplectomorphic to the disc bundle $N_Y(\epsilon_0)$ with the
symplectic form $\omega_{\epsilon_0, A}$. Let $\phi:({\mathcal
N}_{\epsilon_0}(Y), \omega)\to (N_Y(\epsilon_0),\omega_{\epsilon_0, A}) $ be such a symplectomorphism.

\subsubsection{Blowing up via the universal construction}
Let $i:BL\subset \mathbb P^{k-1}\times \mathbb C^k$ be the incidence relation. Then the projection $\alpha:BL\to \mathbb P^{k-1}$ makes $BL$ into the (holomorphic) tautological line 
bundle.  The map  $\beta:BL \to \mathbb C^{k}$ sending  each fiber of $\alpha$  into the corresponding  one dimensional subspace of $\mathbb C^{k}$ is a (holomorphic) bijection of the complement of
the zero section $BL_0$ of $BL$ with the complement of the origin in $\mathbb C^k$. 
In other words, $BL$ is the complex blowup of $\mathbb C^k$ at the origin. 

Let $\Omega$ be the standard $U(k)$ invariant symplectic form on $\mathbb P^{k-1}$ and recall that  $\tau$ is the standard symplectic form on $\mathbb C^k$. 
By Theorem 5.1 in \cite{GS}, for $\epsilon>0$, the form $\omega_{\epsilon}=i^*(\epsilon \, \hbox{pr}_1^*\Omega+\hbox{pr}_2^*\tau)$ defines    a $U(k)-$invariant symplectic structure  on $BL$. 
$(BL, \omega_{\epsilon})$ is the $\epsilon-$symplectic blowup of $(\mathbb C^k, \tau)$. 

Given $\epsilon_0>0$, let $BL_{\epsilon_0}$ be the inverse image of  the $\epsilon_0$ disk  in $\mathbb C^k$ under $\alpha$. 
Note that $(BL_{\epsilon_0}, \omega_{\epsilon})$ is a $U(k)$ symplectic manifold. 
Pick $\epsilon<\epsilon_0$ and apply the universal construction to the principal $U(k)$ bundle $P$ over $Y$ and $(F, \omega_F)=(BL_{\epsilon_0}, \omega_{\epsilon})$, the resulting symplectic manifold $\tilde N^{\epsilon}_Y(\epsilon_0)$ is the local $\epsilon-$symplectic blowup along $Y$. 
Note that the zero section $BL_0$ of $\alpha: BL\to \mathbb P^{k-1}$  is $U(k)-$invariant and so defines a codimension 2 submanifold $\tilde N^{\epsilon}_Y(0)\subset \tilde N^{\epsilon}_Y(\epsilon_0)$. 
By Theorem 5.4 in \cite{GS}, outside the zero section $BL_0$ of $BL$, the symplectic form $\omega_{\epsilon}$ is equivalent to the standard symplectic form $\tau$ on $\mathbb C^k-D_{\epsilon}$. 
In particular, $\tilde N^{\epsilon}_Y(\epsilon_0)\setminus \tilde N^{\epsilon}_Y(0)$ and $N_Y(\epsilon_0)\setminus Y$ are symplectomorphic.

Therefore we can extend the local $\epsilon-$symplectic blowup   via $\phi:({\mathcal
N}_{\epsilon_0}(Y), \omega)\to (N_Y(\epsilon_0),\omega_{\epsilon_0, A}) $
 to obtain the (global) $\epsilon-$symplectic blowup of $X$ along $Y$.

\begin{definition}\label{blowup}
Given $\epsilon, A, \phi$, we call 
$$\tilde X_{\epsilon, A, \phi}=(X-{\mathcal
N}_{\epsilon}(Y))\cup_{\phi}  \tilde N^{\epsilon}_Y(\epsilon_0) $$ 
the $(\epsilon, A, \phi)-$blowup of $X$ along $Y$.

The codimension $2$ divisor $\tilde N^{\epsilon}_Y(0) $ is called the (symplectic) exceptional divisor.
The exceptional divisor is a linear projective space bundle and the symplectic form on it is called a standard form. 
\end{definition}

\begin{remark} \label{standard} Notice that
the construction depends on $\epsilon$, the connection $A$ and the
symplectomorphism $\phi:{\mathcal N}_{\epsilon_0}(Y)\to
N_Y(\epsilon_0)$. However, as remarked on p. 250 in \cite{McS}, given
two different choices $A, \phi$ and $A', \phi'$, for sufficiently
small $\epsilon$, the resulting symplectic forms are isotopic.
We will often simply denote the blowup by $\tilde X$ ignoring various
choices.

\end{remark}

Observe that we can
define a map
\begin{equation}p:\tilde X\to X\label{p}
\end{equation}
 which is identity away from ${\mathcal N}_{\epsilon_0}(Y)$. Such
a map can be constructed by identifying ${\mathcal
N}_{\epsilon_0}(Y)-{\mathcal N}_{\epsilon}(Y)$ with the deleted
neighborhood ${\mathcal N}_{\epsilon_0}(Y)-Y$ using a diffeomorphism
from $(\epsilon, \epsilon_0)$ to $(0, \epsilon_0)$. Such a $p$ is
not unique, but the induced maps  $p_*$ and $p^*$ on homology and
cohomology are the same for different choices.

\subsection{Up to deformation via a simple birational cobordism}

Note that the symplectic blowing up construction is a local $U(k)$ equivariant construction. 
It is shown in \cite{GS} that there is an $S^1$ equivariant description. 

We first recall the notion of symplectic birational cobordism. This notion is based on symplectic cobordism introduced in \cite{GS} (also see \cite{HLR}). 
\begin{definition}
Two symplectic manifolds $(X, \omega)$ and $(X', \omega')$ are birational cobordant  if and only if there are finite number of symplectic manifolds $(X_i, \omega_i)$, $0\le i \le k$, with $(X_0, \omega_0)=(X, \omega)$ and $(X_{k+1}, \omega_{k+1})=(X', \omega')$, and for each $i$, $(X_i, \omega_i)$ and $(X_{i+1}, \omega_{i+1})$ are symplectomorphic to symplectic reductions of a semi-free Hamiltonian $S^1$ symplectic manifold $W_i$ (of two dimensions higher) at different regular  levels or  simple critical levels.
\end{definition}

Given a semi-free Hamiltonian $S^1$ manifold with moment map $\Phi$, 
a critical level of $\Phi$ is called simple if the corresponding set $W$ of critical points is connected and the signature of $\Phi$ is either $(2p, 2)$ or $(2, 2q)$ for some positive values of $p$ and $q$. It follows from Theorem 10.2 in \cite{GS} that the reduction at a simple critical value is still a smooth symplectic manifold. 
The change of the symplecto-diffeotype of  the reductions when passing through a simple critical value is called a simple symplectic birational cobordism. 
By Theorem 11.2 in \cite{GS}, the change of the symplecto-diffeotype of  the reductions when passing through a  critical value of signature $(2p, 2q)$  is the composition of two  simple symplectic birational cobordisms,  one of signature $(2p, 2)$ and the other of signature $(2, 2q)$. 

In Section 12 of \cite{GS},  for a symplectic submanifold $Y\subset (X, \omega)$, a semi-free Hamiltonian $S^1$ manifold $X_I$ with proper moment map $\Phi:X_I\to I$   is constructed,   where $0$ is a simple critical value of $\Phi$ with reduction $(X, \omega)$ and the reduction at small $\epsilon\geq 0$  is
the $\epsilon-$symplectic blowup of $(X, \omega)$ along $Y$ (including the case $\epsilon=0$). Further, Theorem 13.1 of \cite{GS} establishes the uniqueness of $X_I$ for a sufficiently small interval $I$ about $0$. 
In particular, for small $\epsilon$, symplectic blowing up/down is unique up to symplectomorphism.

Since the change of the symplecto-diffeotype of  the reductions when not passing through a  critical value is an integral deformation, we 
have the following conclusion: 
\begin{theorem} \cite{GS}
 Up to integral deformation, a simple birational cobordism is  the same as a symplectic blowing up/down. 
\end{theorem}

This  motivates us to also describe a symplectic blowup up to integral deformation via symplectic cut. 

\subsection{Up to deformation from  symplectic cut} 
Now we apply the local $S^1$ equivariant symplectic cut construction by Lerman  to construct the blowup along $Y$. This observation was already mentioned in Remark 1.3 in \cite{Le}
when $(X, \omega)$ is globally Hamiltonian with $Y$ as the maximal submanifold. 

We provide some details in the general case to show that symplectic cutting a small neighborhood of $Y$ gives rise to a  simple birational cobordism. 
As pointed out in the introduction, this description  introduces a linear $(\mathbb P^k, \mathbb P^{k-1}, \mathbb P^0)-$bundle triple, which leads to a ratio constraint. 
Secondly, it leads to a  construction of the symplectic blowing down up to deformation  via normal connected sum.

\subsubsection{Symplectic cut}
 Suppose that $X_0\subset X$ is an open codimension
zero submanifold with a Hamiltonian $S^1-$action. Let $H:X_0\to
\mathbb R$ be a Hamiltonian function with a regular value $\epsilon$. If
$H^{-1}(\epsilon)$ is a separating hypersurface of $X_0$, then we obtain
two connected manifolds $X_0^{\pm}$ with boundary $\partial
X_0^{\pm}=H^{-1}(\epsilon)$, where the $+$ side corresponds to $H<\epsilon$.
Suppose further that $S^1$ acts freely on $H^{-1}(\epsilon)$. Then the
symplectic reduction $Z=H^{-1}(\epsilon)/S^1$ is canonically a symplectic
manifold of dimension $2$ less. Collapsing the $S^1-$action on
$\partial X^{\pm}=H^{-1}(\epsilon)$, we obtain closed smooth manifolds
$\overline{X}^{\pm} $ containing respectively real codimension $2$
submanifolds $Z^{\pm}=Z$ with opposite normal bundles. Furthermore
$\overline {X}^{\pm}$ admits a symplectic structure $\overline
\omega^{\pm}$ which agrees with the restriction of $\omega$ away
from $Z$, and whose restriction to $Z^{\pm}$ agrees with the
canonical symplectic structure $\omega_Z$ on $Z$ from symplectic
reduction. The pair of symplectic manifolds $(\overline {X}^{\pm},
\omega^{\pm})$ is called the symplectic cut of $X$ along
$H^{-1}(\epsilon)$ (\cite{Le}).

This is neatly shown by 
considering  the standard symplectic structure
$\sqrt {-1}dw\wedge d\bar w$  on  $\mathbb C$ and 
 two Hamiltonian actions of $S^1$ on 
 $X_0\times \mathbb C$ 
with  the  product symplectic structure, 
 where $S^1$ acts on $\mathbb C$ by complex
multiplication by $z$ and $z^{-1}$ respectivley.
The extended Hamiltonian is therefore $H_{\pm}=H \pm |w|^2$ and $(\overline {X}^{\pm}, \omega^{\pm})$ is the reduction at $\epsilon$ with respect to $H_{\pm}$.

\begin{lemma} \label{deformation}
Suppose the moment map image of the $S^1-$action on $X_0$ is $(c, d)$, $c< a< b<d$ and  there are no critical values in the interval $I=[a, b]$. 
Let $(K^t, \omega_t)$ be the symplectic cut on the $-$ side at  $t\in [a, b]$. 
Then  $(K^a, \omega_a)$ and $(K^b, \omega_b)$  are integral deformation equivalent.
\end{lemma}

\begin{proof}{proof}
In the special case that $(X, \omega)$ is   a global $S^1-$manifold, for $\Phi=H-|w|^2$, $\Phi^{-1}(I)/S^1$ is a cobordism with no critical levels. 
So we have an integral deformation. 

In the general local case, $H^{-1}([a-\ep, b])$ is equivariantly $P\times [a-\ep, b]$ by the equivariant version of
the coisotropic embedding theorem, where  $P$ is a principal $S^1-$bundle. 
Consider  a smooth family of diffeomorphisms  $\xi_t: [a-\ep, a]\to [a-\ep, t]$ for $t\in [a, b]$,
which is identity near $a-\ep$.

We use $\xi_t$ to define a family of $S^1-$equivariant diffeomorphisms $\Theta_t: X-H^{-1}((a, d))\to X-H^{-1}((t, d))$ such that 
(i) $\Theta_t$ is the  identity on $X- H^{-1}([a-\ep, d])$, and 
(ii)  $\Theta_t: (s, p)\mapsto (\xi_t(s), p)$ on $H^{-1}([a-\ep, a])$. 
Each  $\Theta_t$ descends to a  diffeomorphism from the $a-$cut to the $t-$cut, still denoted by $\Theta_t$. 
The family of symplectic forms $\Theta_t^* \omega_t$ is then the  desired deformation. 
\end{proof}

\subsubsection{Symplectic cutting  a standard neighborhood of $Y$}
The ${\epsilon_0}$ neighborhood ${\mathcal
N}_{\epsilon_0}(Y)$ of $Y$ carries a $U(k)$ action from the identification $\phi:{\mathcal N}_{\epsilon_0}(Y)\to
N_Y(\epsilon_0)$. 
Consider the Hamiltonian $S^1-$action on $$X_0={\mathcal
N}_{\epsilon_0}(Y)$$  which corresponds to  the complex multiplication on $N_Y(\epsilon_0)$.  The moment map
$H(u)=|\phi(u)|^2, u\in \mathcal N_{\epsilon_0}(Y),$
where $|\phi(u)|$ is the norm of $\phi(u)$ considered as a vector in
 a fiber of the Hermitian bundle $N_Y$. Here $X_0\times \mathbb C={\mathcal
N}_{\epsilon_0}(Y)\times \mathbb C$ is identified via $\phi$ with 
$${
N}_Y({\epsilon_0})\oplus {\mathbb C}.$$ 
Fix $\epsilon$ with
$0< \epsilon<\epsilon_0$ and  consider
 the hypersurface
$Q=H^{-1}(\epsilon)$ in $X_0$ corresponding to the sphere bundle of $N_Y$
with radius $\epsilon$. 
We cut $X$ along $Q$ to obtain two closed symplectic manifolds
$(\overline X^+, \omega^+)$ and $(\overline X^-, \omega^-)$, each containing a codimension 2 symplectic submanifold $Z^{\pm}=Q/S^1$. 

 $Z^{\pm}$ is diffeomorphic to the  projectivization $\mathbb P_s(N_Y)$ of $N_Y$. 
 It is important to remember that 
 $$(Z^+, \omega^+|_{Z^+})=(Z^-, \omega^-|_{Z^-}).$$

\subsubsection{The $\overline X^-$ side as blowup}
We apply the birational cobordism in \cite{GS} to  show that $\overline X^-$, which is constructed via  a local $S^1-$symmetry,  gives an alternative construction of  the blowup of $X$ along $Y$ (which  
uses a local $U(k)-$symmetry).

First of all,  $\overline X_0^-=H_-^{-1}(\epsilon)/S^1$ is the reduction at $\epsilon$ with respect to the Hamiltonian
  $$H_-(u, w)=|\phi(u)|^2-|w|^2.$$ 
Since  $H_-$ has only one simple critical value $0$ whose signature is $(2, 2k)$ which is connected when $Y$ is, we know $\overline X_0^-$ is the $\epsilon-$blowup of $X_0={\mathcal
N}_{\epsilon_0}(Y)$ along $Y$ by Theorem 13.1 of \cite{GS} (while the reduction at a small negative value is  ${\mathcal N}_{\epsilon_0}(Y)$ with a deformed form). 
The following observation  is essentially  pointed out in \cite{Le}. 

\begin{lemma} \label{cut=blowup} The symplectic manifold $(\overline X^-, \omega^-)$ via symplectic cut is the  $\epsilon_0-$blowup of $X$ along $Y$,
where  $Z^-$ is the exceptional divisor.
\end{lemma}

\begin{proof}
We apply the construction in Section 12 in \cite{GS} to construct a global cobordism $(X_I, \Phi)$ with a simple critical value $0$. We divide $X$ as two open sets $X-\mathcal N_{\epsilon'}(Y)$ and $X_0$ with $\epsilon'<\epsilon_0$, construct birational cobordisms on each piece and glue them together. For the part $X-\mathcal N_{\epsilon'}(Y)$, we choose the trivial cobordism $I\times S^1\times (X-\mathcal N_{\epsilon'}(Y))$ with the circle action the translation on the $S^1$ factor. For the  $X_0$ piece, we use the   simple cobordism as  in the  paragraph above. 
\footnote{If $X_0$ can be taken to be $X$ then we are done. This is the case in Remark 1.3 in \cite{Le}. }

By applying the  ``real blowing up" trick  in \cite{GS} to our local symplectic cut cobordism over $X_0$,
 we can glue two birational cobordisms together along $I\times S^1\times (\mathcal N_{\epsilon_0}(Y)-\mathcal N_{\epsilon'}(Y))$ to get a global cobordism $(X_I, \Phi)$ with a simple critical value $0$, where reductions at negative values are $X$ and positive values are $\overline X^-$. Now the lemma follows from Theorem 13.1 of \cite{GS}. 
\end{proof}

\subsubsection{The $\overline X^+$ side}
Let us examine the $\overline X^+$ side. First of all, we have 
$\overline X^+=\overline X_0^+=H_+^{-1}(\epsilon)/S^1$ where $H_+(u, w)=|\phi(u)|^2+|w|^2$. Notice $H_-|_{{\mathcal N}_{\epsilon_0}(Y)\oplus 0}=H_+|_{{\mathcal N}_{\epsilon_0}(Y)\oplus 0}=H$.

\begin{lemma} \label{triple}
  $\overline X^+$ is diffeomorphic to the  projectivization $\mathbb P_s(N_Y\oplus {\mathbb C})$ of
$N_Y\oplus {\mathbb C}$. It has two symplectic submanifolds,  
 the codimension 2 submanifold $Z^+=\mathbb P_s(N_Y\oplus 0)$ and  a copy of  $Y$ which is the infinity  section ${\mathbb P}_s(0\oplus {\mathbb C})$.

$\bullet$ Via  the trivial $\mathbb C-$summand, there is an embedding $$N_Y\to \mathbb P_s(N_Y \oplus \mathbb C), \quad  v\to l(v, 1).$$
 Under this embedding,     the zero section $Y$ of $N_Y$ is mapped to $\mathbb P_s(0\oplus \mathbb C)=\{l(0, 1)\}$. 
Thus the normal bundle to the  section
$\mathbb P_s(\mathbb C)$ is  
$N_Y$. 

$\bullet$    $\overline X^+$ inherits a semi-free $S^1-$action  from $X$ with $Z^+$ and $Y$ as the fixed loci. 
This action lifts to the bundle $N_Y\oplus \mathbb C$, complex multiplication on $N_Y$ and trivial action on $\mathbb C$.

$\bullet$ The symplectic forms on $\overline X^+$ and the restriction to $Z^+$ are standard.
\end{lemma}

All the statements are clear.

The symplectic section $\mathbb P_s(\mathbb C)=\mathbb P_s(0\oplus \mathbb C)$ imposes constraint on the possible symplectic structures on the $\mathbb P^{n-1}-$bundle $Z^-$. 
In the next section we search for constraints in the case that $Y$ is a $2-$manifold.

\subsubsection{Symplectic cutting a general neighborhood}
Suppose $W$ is a neighborhood of $Y$ containing the  $U(k)$ neighborhood 
${\mathcal
N}_{\epsilon_0}(Y)$ and $W$ has a Hamiltonian  $S^1$ action with  the Hamiltonian function $H$  extending $|\phi(u)|^2$. 

\begin{lemma} \label{bu deformation}
If  $\alpha\geq \epsilon_0>\epsilon$ is in the interval $H(W)$ and we cut the $\alpha$ neighborhood of $Y$, then 

$\bullet$ $(\overline X^-, \omega^-)$ is deformation to the $\epsilon-$blowup.

$\bullet$  The symplectic form on $Z^{\pm}$ is almost standard. 

$\bullet$ $(\overline X^+, Z^+, \mathbb P_s(\mathbb C); \omega^+)$ is a matching triple of $(\overline X^-, Z^-)$.

\end{lemma}
\begin{proof}
It follows from Lemma \ref{deformation} that $\overline X^-$ is deformation to the $\epsilon-$blowup.

Recall   from Definition \ref{matching triple} that 
a matching   triple of $(\overline X^-, Z^-)$ is a linear   $(\mathbb P^{k}, \mathbb P^{k-1}, \mathbb P^0)-$bundle triple  $(K, D', S; \Omega)$ over $Y$ satisfying 
\begin{enumerate}

\item $D'$ is diffeomorphic to $Z^-$;

\item $e(N_{D'})=-e(N_{Z^-})$;

\item $\Omega|_{D'}$ is almost standard and matches with $\omega^-|_{Z^-}$.

\item $\Omega$ is $S^1-$invariant with respect to a semi-free  $S^1-$action, where $D'$ and $S$ are  exactly the fixed loci. 
\end{enumerate}
By Lemma \ref{triple} the triple $(\overline X^+, Z^+, \mathbb P_s(\mathbb C); \omega^+)$  clearly satisfies these conditions. 
\end{proof}

\subsection{Blowing down  via the $S^1-$equivariant  fiber sum}
Blowing  down is the inverse operation of blowing up. 
Suppose $(M, \omega)$ has a topological  exceptional divisor $\pi:D\to Y$. 
We will state a criterion to   blow down $D$ symplectically up to deformation using normal connected sum.

Given two symplectic manifolds containing symplectomorphic
codimension 2 symplectic submanifolds with opposite normal bundles,
the normal connected sum operation in \cite{G} and \cite{MW}  produces a new symplectic
manifold by identifying the tubular neighborhoods. As pointed out in \cite{Le}, the normal connected sum operation  or the
fiber sum operation is the inverse operation of the symplectic cut.

Notice that we can apply the normal connected sum operation to the
pairs
$(\overline
X^{\pm}, \omega^{\pm}|Z^{\pm})$ to recover 
$$(X, \omega)=(\overline
X^+, \omega^+)\#_{Z^+=Z^-}(\overline X^-, 
\omega^-).$$
A matching triple of $(\overline
X^{-}, Z^{-})$ also satisfy  
the conditions to perform a symplectic sum with $(\overline
X^{-}, Z^{-})$. 
Topologically, the new manifold obtained is the blow down of $M$. Moreover, we have the following counterpart of Lemma \ref{bu deformation}.

\begin{prop} \label{fiber sum'} Suppose $(M, \omega)$ has a topological  exceptional divisor $D$. 
We can symplectically blow down $(M, \omega)$ along $D$ up to integral deformation if there is  a matching  triple $(K, D', S; \Omega)$. 
\end{prop}
\begin{proof}
Let $X=M\#_{D=D'} K$ be the normal connected  sum of $(M, \omega)$ and $(K, \Omega)$ along $D=D'$. 

By the universal construction we can  choose a tubular neighborhood $W=\mathcal N_{\epsilon}(D)$ of the symplectic submanifold $D$ in $(M, \omega)$
with a semi-free Hamiltonian $S^1-$action, which is  complex multiplication when identifying $W$ with a normal disk bundle over $D$. In particular, $D$ is fixed by this $S^1-$action. 
By the equivariant Darboux-Weinstein Theorem, 
we can glue the two Hamiltonian $S^1-$actions on $K$ and $W$ to get an $S^1-$action  on the open piece $X_0=W\#_{D=D'} K$ of $X$.
Let $H$ be the moment map on $K$ with value $0$ at $S$ on $K$ and $\tau=H(D')$. 
We extend $H$ to $X_0$.

 Pick  a standard $\epsilon_0$-neighborhood of $S$ in $K$, which of course lies in $X$.  
For some $\epsilon <\epsilon_0$, perform the symplectic cut of $X$  along $H^{-1}(\epsilon)$ to get $(\tilde X_1, \tilde D_1)$. Then $(\tilde X_1, \tilde D_1)$ is the symplectic $\epsilon-$blowup of  $(X, S)$ by Lemma \ref{cut=blowup}. 
Note that $M$ is the cut at $\tau$ and $\tilde X_1$ is the cut at $\epsilon$. 
By Lemma \ref{deformation}, $M$ and $\tilde X_1$ are (integral) deformation from the $S^1$-action on $X_0$. 
\end{proof}

When $Y$ is a $2-$manifold, we will show that a necessary condition for the existence of a matching triple is that $D$ being admissible. 
And when $D$ is admissible, we will construct a  weak matching triple using K\"ahler geometry. 
Moreover, in dimension 6, weak matching triples are actually matching triples.

\section{Symplectic  geometry of projective bundles over a surface}
In this section let $\Sigma$ be a closed, oriented $2-$manifold.

\subsection{Topology of  linear projective bundles $\mathbb P_s(V)$ over a surface}

We introduce the topological type and the normal type for a  linear $\mathbb P^{n-1}-$bundle over $\Sigma$, which is the projectivization 
of a  rank $n$ bundle $V$. 
Over $\Sigma$ the 1st Chern class of a complex vector  bundle $V$ can be identified as an integer, the degree $\deg(V)=\int_{\Sigma} c_1(V)$. 
Another feature is that   $V$ can be decomposed as the direct sum of line bundles. 
Especially, every bundle admits a holomorphic structure.

We use $\mathbb C(i)$ to denote the  topological  complex line bundle with degree $i$. 
Up to  twisting by a line bundle,  every rank $n$ complex vector bundle is of the form 
$V= \mathbb C^{k} \oplus \mathbb C(-1)^{n-k}, 1\leq k\leq n$.
\subsubsection{Topological and normal types}


First observe that  $$\pi_1(PGL(n, \mathbb C))=\pi_1(PU(n))=\mathbb Z_n$$ since   $SU(n)$ is   $n-$fold cover of   $PU(n)$. 
So there are $n$   topologically  distinct  linear $\mathbb P^{n-1}-$bundles over $\Sigma$. 
Since tensoring a line bundle does not change the
projective bundle, we see topological  linear $\mathbb P^{n-1}-$bundles are classified by  $-\deg(V) \pmod n$.

\begin{definition} \label{type}
Suppose $D_V=\mathbb P_s(V)$ is a  linear $\mathbb P^{n-1}-$bundle modeled on a rank $n$ complex vector bundle $V$. 
The topological type of $D_V$   is $$t_n(D_V)=t_n(-\deg V) \in \{0, \cdots, n-1\}$$ i.e. 
 the smallest non-negative integer  satisfying
$$t_n(D_V)\equiv -\deg(V) \pmod n.$$ 
We also have the integer valued  `normal type' of $D_V$, 
$$N_n(D_V)=-\deg(V)\in \mathbb Z.$$
\end{definition}

\subsubsection{Cohomology and homology of $D_V$} 

Suppose $D_V$ is a   linear $\mathbb P^{n-1}-$bundle modeled on $V$.
Consider  the tautological line subbundle  $\Xi\subset \pi^*V$ over $D_V$ and its dual bundle $\Xi^*$.  Here are a few properties of $\Xi$. 

\begin{lemma} \label{Xi}
 $\Xi$ depends on $V$. If we change $V$ by a line bundle $L$ over $\Sigma$, then $\Xi$ is changed to $\Xi\otimes \pi^*L$. 

\begin{enumerate}

\item If $V$ is a line bundle, then  $D_V=\Sigma$ and $\Xi=V$. 

\item  For a subbundle $T\subset V$, $\Xi|_{\mathbb P_s(T)}\subset \pi^*T|_{P_s(T)}$ is the tautological line subbundle over $P_s(T)$. 

\item  $\Xi^*$ is  the normal bundle of $D_V=\mathbb P_s(V)$ in $\mathbb P_s(V\oplus \mathbb C)$. In fact, $\mathbb P_s(V\oplus \mathbb C) \setminus \mathbb P_s(\mathbb C)$
is the total space of $\Xi^*$. Its fiberwise picture is just $\mathbb P^n-{0}$ is biholomorphic to the total space of $\mathcal O(1)$ over $\mathbb P^{n-1}$.

\end{enumerate}
\end{lemma}

The dual line bundle $\Xi^*$ is called the hyperplane line bundle.  Set $$\tau=c_1(\Xi^*).$$
Since $\tau$ restricts to a ring generator of
the de Rham cohomology of the fiber, 
by the Leray-Hirsch principle,  the de Rham cohomology group $H^*(D_V;\mathbb R)$  is the tensor product of 
$H^*(\mathbb P^{n-1};\mathbb R)$ and $H^*(\Sigma;\mathbb R)$. 
As an $H^*(\Sigma;\mathbb R)-$algebra, $H^*(D_V;\mathbb R)$ is generated by $\tau$  subject to  the  relation 
\begin{equation}\label{Bott-Tu}  \sum   \,\, \pi^*c_j(V)\, \tau^{n-j}=0.
\end{equation}
In fact, this is the defining relation of Chern classes (see {\it e.g.} \cite{BT}).

Let
$F$ denote the homology class of the fiber, as well as its Poincare dual in $H^2(D_V;\mathbb Z)$. 
Let $l$ be the homology class of a line in the fiber.   Then clearly 
$$\langle  F, l \rangle =0,  \quad 
\langle \tau, l\rangle =1.$$

\begin{lemma}\label{cohomology ring} 
Suppose $V$ is a rank $n\geq 2$ vector bundle over $\Sigma$ with degree $d$. Then $H^2(D_V;\mathbb R)$ is $2$ dimensional with $\{F, \tau\}$ as a basis. 
 And the  even cohomology ring structure is described by 
$$\tau^{n-1}\cdot F=(\tau|_F)^{n-1}=1, \quad F\cdot F=0, \quad  \tau^n=-\int_{\Sigma} c_1(V)=-d.$$
There exists a unique degree $2$ integral homology class  $\eta$ such that $\langle \tau, \eta \rangle=0$ and $\langle F, \eta\rangle =1$. 
In particular,    $l$ and $\eta$ are an integral  basis of $H_2(D_V; \mathbb Z)$. 

If $s$ is a section of a line subbundle $L$, then $$\langle \tau, s\rangle =-\deg(L).$$ 
\end{lemma}

\begin{proof}
The relation  $\tau^n=-\int_{Y} c_1(V)=-d$ follows from \eqref{Bott-Tu}  and  the vanishing of $c_j(V)$ for $j>1$.

To construct the class  $\eta$ observe that  $F$ is primitive since $\tau^{n-1}\cdot F=1$.  So there exists  a class  $\eta'$ whose paring with $F$ is $1$. Since $\langle F, l \rangle =0$
and 
$\langle \tau, l\rangle =1$, we get the desired $\eta$ from adjusting $\eta'$  
 by multiple
of $l$ to achieve trivial pairing  with $\tau$. 

The pairing $\langle \tau, s\rangle=-\deg(L)$ follows from the observation that,  under the natural identification of the section $s$ with $\Sigma$ via $\pi$, $\Xi$ restricted to $s$ is identified with $L$. 
A consequence is that  $\eta$ is geometrically represented by  a nowhere zero section  of a trivial line bundle, which always exists. 
\end{proof}

\begin{lemma}\label{cubic}
For a bundle $V$ with rank $n$ and degree $d$, a class $u=x\tau+yF$ is in the forward  cone $\{q\in H^2(D_V; \mathbb R)| q^n>0, \langle q, l\rangle>0\}$ if and only if
$x>0$ and $  \frac{y}{x}>\frac{d}{n}.$ 

The ratio of $u$ is given by $\rho_{\pi}(u)=-d +n \frac{y}{x}.$
If $\langle u, \eta\rangle>0$, then $\rho_n(u)> -d$. 
\end{lemma}
\begin{proof}
The description of the forward cone  follows from $$\langle u, l\rangle= \langle x\tau+yF, l\rangle =x>0,$$
$$u^n=(x\tau+yF)^n=x^n\tau^n+nx^{n-1}y(\tau^{n-1}\cdot F)=x^{n-1}(-dx+ny)>0.$$ 
It follows that the ratio of $u$ is given by
$$\rho_{\pi}(u)=\frac{u^n}{\langle u, l\rangle^n}=\frac{ x^{n-1}(-dx+ny)} {x^n}=     -d +n \frac{y}{x}.$$
Note that $\langle u, \eta\rangle =y$. If $\langle u, \eta\rangle >0$, then $\rho_{\pi}(u)=-d +n y/x> -d$. 
\end{proof}

We end this subsection with a remark on smooth, linear, symplectic  and holomorphic $\mathbb P^n-$bundles. 

\begin{remark} 
Every complex bundle over a Riemann surface splits smoothly into a direct sum of complex line bundles and hence can be made holomorphic. 
Therefore any linear $\mathbb P^n-$bundle over a Riemann surface is of the form $\mathbb P_s(\E)$ of a holomorphic vector bundle $\E$.
We will study the K\"ahler geometry of such bundles in the next subsection. 

Over an algebraic variety, a  holomorphic  $\mathbb P^n-$bundle arises from a holomorphic vector bundle   (cf. Exercise  7.10  in Hartshorne's
Algebraic Geometry).  
Therefore if a smooth $\mathbb P^n$ bundle over surface can be made  a holomorphic $\mathbb P^n$ bundle, then it is linear. 

A symplectic fibration with fibers $(F, \sigma)$  is a fibration with structure group $Symp(F, \sigma)$. 
Since the fibers of $D$ are compact, a  closed form $\omega$ on $D$ which is symplectic along fibers gives rise to a symplectic fibration structure (\cite{McS}). 

When $n=2$,  if a smooth $\mathbb P^n$ bundle can be made symplectic, then  the structure group is Symp$(\mathbb P^2, \omega_{FS})$. Since Symp$(\mathbb P^2, \omega_{FS})$ is homotopic to $PU(3)$, such a bundle is linear.

 When $n=1$, the classification of symplectic ruled surface shows that all smooth $\mathbb P^1$   bundles are linear.

\end{remark}


\subsection{Holomorphic projective  bundles  $\mathbb P(\E)$} 
\subsubsection{The projective bundle $\mathbb P(\E)$ of quotient line bundles} 

We switch to the quotient bundle convention in algebraic geometry. 
Let us begin with a general setting. Let $X$ be an algebraic variety, and $\E$  a holomorphic vector bundle of rank $r$ on $X$. 
We use the Grothendick convention of projectivization and 
 denote by $\pi: \mathbb P(\E)\rightarrow X$ the projective bundle of one-dimensional quotients of $\E$. More algebraically, $\mathbb P(\E)=\hbox{Proj}_{\mathcal O_X}(\hbox{Sym}(\E))$,  where $\hbox{Sym}(\E)=\oplus_{m>0} s^m(\E)$ is the symmetric algebra of $\E$.

The projective bundle $\mathbb P(\E)$ carries the Serre line bundle $\mathcal O_{\mathbb P(\E)}(1)$, which is the tautological quotient of $\pi^*\E$: $$\pi^*\E\rightarrow \mathcal O_{\mathbb P(\E)}(1)\rightarrow 0.$$ 
When restricted to each fiber,  $\mathcal O_{\mathbb P(\E)}(1)$  is the hyperplane  bundle $\mathcal O(1)$ on $\mathbb P^{n-1}$. 
$\mathcal O_{\mathbb P(\E)}(1)$ is the universal quotient bundle in the following sense. 
Let $p: Z\rightarrow X$ be any morphism. Then,  a morphism $f: Z\rightarrow \mathbb P(\E)$ over $X$ is equivalent to  a quotient line bundle $p^*\E\rightarrow \mathcal L\rightarrow 0.$ Under this correspondence, $\mathcal L=f^*\mathcal O_{\mathbb P(\E)}(1)$.

Since $X$ is projective, $\mathcal O_{\mathbb P(\E)}(1)$ is represented by a divisor and we write $\xi$ for its first Chern class. 
As an $H^*(X;\mathbb R)-$algebra, $H^*(\mathbb P(\mathcal E);\mathbb R)$ is generated by $\xi$  subject to  the Grothendick relation 
\begin{equation}\label{Grothendick}  \sum  (-1)^{j} \,\, \pi^*c_j(\mathcal E)\, \xi^{n-j}=0.
\end{equation}
We call $\E$  {\it ample} (resp. {\it nef}) if  $\mathcal O_{\mathbb P(\E)}(1)$ is so.

We  now assume $X=\Sigma$ is a Riemann surface. 
We need the following result of Hartshorne in \cite{nagoya}  later.

\begin{theorem}\label{Hartshorne}
For a vector bundle $\E$ on a Riemann surface $\Sigma$, 
   $\mathcal O_{\mathbb P(\E)}(1)$ is nef (resp. ample) if and only if $\E$ and every quotient bundle of $\E$ has non-negative (resp. positive) degree.
Especially, suppose $\E$ is semistable, then $\mathcal O_{\mathbb P(\E)}(1)$    is nef (resp. ample) if it has non-negative (resp. positive) degree.
\end{theorem}

\subsubsection{Cohomology and homology of $\mathbb P(\E)$}
We still use $F, l, \eta$ for $\mathbb P(\E)$ and 
summarize the results on the real cohomology and integral homology of $\mathbb P(\E)$, $H^*(\mathbb P(\E))$ and $H_2(\mathbb P(\E);\mathbb Z)$.

\begin{lemma}\label{cubic-kahler}
Suppose $\E$ is  a holomorphic bundle over a Riemann surface $X$ with rank $n$ and degree $d=\deg(\E)$. 
Then 
\begin{enumerate}
\item 
$\mathbb P(\E)=\mathbb P_s(\E^*) \quad \hbox{and} \quad \deg(\E)=-\deg(\E^*).$
Under the identification  $\mathbb P(\E)=\mathbb P_s(\E^*)$, the $\xi$ class corresponds to the $ \tau$ class of $\mathbb P_s(\E^*)$.

\item $H^2(\mathbb P(\E);\mathbb Z)$ is generated by $\xi$ and the (Poincar\'e dual of) the fiber class $F$.

\item The  even cohomology ring structure of $\mathbb P(\E)$  is described by 
$$\xi^{n-1}\cdot F=(\xi_F)^{n-1}=1, \quad F\cdot F=0, \quad  \xi^n=d.$$

\item  $H_2(\mathbb P(\E);\mathbb Z)$ is generated by $l$ and $\eta$, with the pairing with $H^2(\mathbb P(\E);\mathbb Z)$ given by 
$$\langle  F, l \rangle =0,  \quad 
\langle \xi, l\rangle =1,\quad   \langle  F, \eta \rangle =1,  \quad 
\langle \xi, \eta\rangle =0       .$$

\item  For a class $u=x\xi+yF\in H^2(\mathbb P(\E); \mathbb R)$,
$$u^n=x^{n-1}(dx+ny) \quad  \hbox{and} \quad   \langle u, l\rangle=\langle x\xi+yF, l\rangle=x.$$
Hence $u$  is in the forward  cone $\{q\in H^2(\mathbb P(\E); \mathbb R)| q^n>0, \langle q, l\rangle>0\}$ if and only if
$$x>0\quad \hbox{and} \quad  \frac{y}{x}>-\frac{d}{n}.$$

\item   The ratio of $u=x\xi+yF$ in the forward cone is given by $$\rho_{\pi}(u)=d +n \frac{y}{x}.$$

\item  If $u$ is in the forward cone and $\langle u, \eta\rangle>0$, then $\rho_n(u)> d$. 
\end{enumerate}
\end{lemma}

\begin{proof} For (1) just notice that 
a line subbundle $\mathcal L$ of $\E^*$ determines a hyperplane $\hat {\mathcal L}$ of $\E$, which in turn gives rise to a quotient line bundle $\mathcal L^*$ of $\E$. 

 (2), (3) and  (4) follow from (1) and Lemma \ref{cohomology ring}.

(5), (6) and (7) follow from (1) and Lemma \ref{cubic}. 
\end{proof}
\subsubsection{The curve cone, multi-sections and symmetric powers}

The curve cone $NE(\mathbb P(\E))\subset H_2(\mathbb P(\E))$  is spanned by the classes of all effective $1-$cycles. An irreducible curve is either contained in a fiber,  or it intersects the fiber divisor at zero-cycles. For the first case, the classes are generated by $l$ which is the generator of the second homology of a fiber.  

For the latter case, it is a multi-section $Z$. Suppose $Z$ is an $m-$section. Then $\langle F, [Z]\rangle =m$, and the restriction of the projection $\pi|_Z:Z\to X$ is a degree $m$ ramified covering  of $X$. Let $\iota: Z\rightarrow \mathbb P(\E)$ denote the inclusion.
By the universal property of  $\mathcal O_{\mathbb P(\E)}(1)$,     we have  the  quotient line bundle over $Z$, 
$$\pi|_Z^*\E\rightarrow \iota^*\mathcal O_{\mathbb P(\E)}(1)\rightarrow 0.$$

\begin{lemma} \label{multisection}
Let $p: Z\rightarrow X$ be a ramified covering of algebraic curves of degree $m$, and $\E$ a vector bundle over $X$. Then each quotient line bundle $p^*\E \rightarrow \mathcal L \rightarrow 0$ over $Z$ gives rise to a quotient line bundle $\mathcal M_{\mathcal L}$ of $s^m\E$ over $X$ with  degree $\deg \mathcal L$. 
\end{lemma}

\begin{proof} First notice that any fiber $\mathcal L|_{y}$ could be viewed as a quotient subspace of $ \E|_{p(y)}$. For any $x\in X$, we could choose $m$ $1$-dimensional vector quotient spaces of $\E_x$ by taking  $\mathcal L|_{y_i}$  where $y_i\in p^{-1}(x)$. This may count with multiplicities if it is a ramified point.  Then the line Sym$(\mathcal L_{y_1}\otimes ...\otimes \mathcal L_{y_m})$ is
a $1$-dimensional quotient space in the symmetric power $s^m({\E})|_x$. Hence we obtain  a quotient line bundle of $s^m\E$, which is holomorphic since the transition functions are polynomials of that of $\mathcal L$. 

It is clear from the construction that a section $s$ of $\mathcal L$ would give rise to a section $s'$ of  $\mathcal M_{\mathcal L}$. Moreover,  the zero locus
of $s$  counted with multiplicity  would exactly  correspond to that of $s'$. Hence the degree of the quotient bundle is preserved.
\end{proof}

\begin{remark} Note that  a quotient line bundle  of $s^m({\mathcal  E})$ may not correspond to a $m-$section of ${\mathbb P(\E)}$. This is true if the rank of ${ \E}$ is 2. The simple algebraic fact is a symmetric tensor is not always decomposable, i.e. of the form sym$(v_1\otimes ...\otimes v_m)$. \end{remark}

By Lemma \ref{multisection} an $m-$section $Z$  induces a quotient line bundle  of $s^m\E$, which we denote by $\mathcal M_Z$.
The homology class  of $Z$ is
determined by  $\deg \mathcal M_Z$ as follows.

\begin{lemma} \label{m-section pairing}
Suppose $\iota:Z\to \mathbb P(\E)$ is an $m-$section with the associated quotient line bundle $\mathcal M_Z$ of $s^m\E$.  
Then  $\langle \xi, [Z]\rangle =\deg \mathcal M_Z$ and 
 hence 
 \begin{equation} 
 [Z]=   \deg(\mathcal M_Z) l +  m\eta.
 \end{equation}
 \end{lemma}

 \begin{proof}
 By Lemma \ref{multisection},
 $$\deg \mathcal M_Z=\deg \iota^*\mathcal O_{\mathbb P(\E)}(1)=
 \int_Z c_1(\iota^*\mathcal O_{\mathbb P(\E)}(1))=\int_Z \iota^*\xi=\langle \xi, \iota_* [Z]\rangle.$$
  The last statement follows from the first and 
 $\langle F, [Z]\rangle =m$. 
\end{proof}

\subsubsection{Semi-stable bundles and direct sums}

Recall that the slope of a vector bundle $\E$  is the ratio $\mu(\E)=\deg(\E)/ \hbox{rank} (\E)$. 
$\E$ is called semistable if every subbundle $\mathcal F\subset \E$ satisfies $\mu(\mathcal F)\leq \mu(\E)$. 
Equivalently, $\E$ is  semistable if every quotient bundle $\mathcal G= \E/\mathcal F$ satisfies $\mu(\mathcal G)\geq \mu(\E)$.

\begin{lemma} \label{degree}
If $\V$ is a rank $r$ semi-stable bundle, then $s^m\V$ is also semi-stable with 
$\mu(s^m\V)=m \,\,\mu (\V).$
\end{lemma}
\begin{proof} 
A basic fact about semi-stable bundles over curves  is due to Hartshorne \cite{nagoya}: 
The symmetric powers of a semi-stable bundle over a smooth curve are semi-stable.

Since both bundles are semi-stable, the slope is determined by the rank and the degree of the
bundles themselves. 
For $m\geq 0$ one has that  $\pi_*\mathcal O_{\mathbb P(\V)}(m)=s^m\V.$
We note that   $${\rm rank} \,\,s^m \V= {m+r-1 \choose m}, $$
which is the number of $r-$variable monomials of degree $m$, and  $$c_1(  s^m\V) ={m+r-1 \choose m-1}c_1  (\V),$$
which can be verified by the splitting principle. 
By the formulae above, $$\mu(s^m\V)=\frac  { \deg (s^m\V)}   {  {\rm rank} \,\,s^m({\V})} = m \frac{\deg (\V)}{{\rm rank} \V}=      m \,\,\mu (\V).$$
\end{proof}

\begin{lemma}\label{slope estimate}
Let $\E=\V_1\oplus \cdots \oplus \V_k$ with 
$\V_1, \cdots, \V_k$  semi-stable bundles and  $\mu(\V_1)\ge \cdots \ge \mu(\V_k)$. Then 
\begin{enumerate}
\item If $\mathcal L$ is a line subbundle of $\E$, then $\deg \mathcal L\le \mu(\V_1)$. 
\item If $\mathcal L$ is a quotient line bundle, then $\deg \mathcal L\ge \mu(\V_k)$. 
\end{enumerate}
\end{lemma}
\begin{proof}
If $\mathcal L$ is a line subbundle, then  $\mathcal O$ is a line subbundle of   $$\E\otimes \mathcal L^{-1}=(\V_1\otimes \mathcal L^{-1})\oplus \cdots \oplus (\V_k\otimes \mathcal L^{-1}).$$ As $$\Gamma((\V_1\otimes \mathcal L^{-1})\oplus \cdots \oplus (\V_k\otimes \mathcal L^{-1}))=\Gamma(\V_1\otimes \mathcal L^{-1})\oplus \cdots \oplus \Gamma(\V_k\otimes \mathcal L^{-1}),$$ we know $\mathcal O$ is a subbundle of at least one of $\V_i\otimes \mathcal L^{-1}$. As each $\V_i\otimes \mathcal L^{-1}$ is semi-stable, we have $\deg (\V_i\otimes \mathcal L^{-1})\ge 0$ for at least one $\V_i$, {\it i.e.}, $$\deg \mathcal L\le \mu(\V_i)\le \mu(\V_1).$$
If $\mathcal L$ is a quotient line bundle, then the dual bundle $\mathcal L^*$ is a line subbundle of $\E^*=\V_1^*\oplus \cdots \oplus \V_k^*$. Each $\V_i^*$ is  a semi-stable bundle of slope $-\mu(\V_i)$ and  $\mu(\V_1^*)\le \cdots \le \mu(\V_k^*)$.  By the line subbundle case, $$\deg \mathcal L^*\le \mu(\V_k^*)=-\mu(\V_k),$$ which is $\deg \mathcal L\ge \mu(\V_k)$.
\end{proof}

When the base is $\mathbb P^1$,  we will need the following lemma.

\begin{lemma} \label{GW'}
Let  $\V= \mathcal O(1)^{k} \oplus \mathcal O(2)^{n-k}$ over $\mathbb P^1$ and $s^*$ the section from  the  quotient line bundle $\mathcal O(1)$. 
Then the GW invariant of $[s^*]$  is nonzero. 
\end{lemma}

\begin{proof}
Observe that  $\mathbb P(\V)$ is Fano so we 
 can apply the calculation in Sections 2 and  5 of  \cite{QR} here. 
By the paragraph above Lemma 2.2 in \cite{QR}, $[s^*]$ is the extremal class $A_2=\xi^{n-1}+(1-c_1)F\xi^{n-2}$ and the unparametrized moduli space $\mathcal M([s^*], 0)$ is compact. 
Since $$c_1=\int_Xc_1( \V) =k+2(n-k)=2n-k,$$ $\xi^n=c_1$ and 
 $-K_{\mathbb P(\V)}=(2-c_1)F+n\xi$, we have  
$$\langle -K_{\mathbb P(\V)}, s^*\rangle = (2-c_1) +(1-c_1)n+n\xi^n=2-c_1+n=2-(n-k).$$
 Thus, by (3.4) in \cite{QR},  the GW dimension of unparametrized moduli space is     $ 2-(n-k) +(n-3)(1-0)=k-1$. 
In fact, from the proof of Lemma 2.3 (ii) and (5.6) in  \cite{QR}   the moduli space $\mathcal M([s^*], 0)$ 
is  $\mathbb P^{k-1}$, which is identified with the space of trivial quotient line bundles of $\mathcal O(1)^k$. And the obstruction bundle is trivial by (5.7) in \cite{QR}. 
\end{proof}

\subsection{K\"ahler cone and restricted K\"ahler cone}

\subsubsection{The K\"ahler cone of $\mathbb P(\E)$}
Let $\E$ still be a holomorphic vector bundle over $\Sigma$. We determine the K\"ahler cone of $\mathbb P(\E)$ for various types of $\E$. In particular, we  establish a  criterion for the K\"ahler cone to be the forward cone. As our objects are projective bundle $\mathbb P(\E)$, and in particular they are projective, so the K\"ahler cone is the real extension of ample cone. Since $H^{2, 0}$ vanishes for $\mathbb P(\E)$, it follows from  Kleiman's cirterion that a  class in $H^2(\mathbb P(\E);\mathbb R)$ is K\"ahler  if and only if it is positive on the closure of the cone of curves.  

A K\"ahler class is in the forward cone since the line class $l$  is in the curve cone. 
Let $Z$ be an effective curve which is an $m-$section. Then $Z$ corresponds to a quotient line bundle $\mathcal M_Z$ of 
$s^m\E$,  and by Lemma \ref{m-section pairing}, 
$[Z]=al+m\eta$ with  $a=\deg \mathcal M_Z$.

\begin{lemma}\label{kah=pos}
Suppose  $u=x\xi+yF$ is in the forward cone.
For an $m-$section $Z$ with $[Z]=al+m\eta$,
$\langle u, [Z]\rangle>0$ if and only if $$\frac{y}{x}>-\frac{a}{m}.$$
Consequently,    $u$  is in the K\"ahler cone if  there is $\alpha>0$ such that every $m-$section $Z$  satisfies   $\frac{a}{m}\geq \alpha$ and 
$\rho_{\pi}(u)>\deg(\E)-{\rm rank}(\E) \alpha$.

In particular,  the  K\"ahler cone
of $\mathbb P({\E})$ is the forward cone if 
\begin{equation}\label{ah} \frac{a}{m}  \geq  \frac{\deg (\E) }{{\rm rank} (\E)}. 
\end{equation}
\end{lemma}

\begin{proof} 
Since the class   $u=x\xi+yF$ is in the forward  cone  we have $x>0$ and 
$dx+ny>0$    by Lemma \ref{cubic-kahler} (5).
The first claim follows from 
$$\langle u, [Z]\rangle=\langle x\xi+yF, al+m\eta\rangle=ax+my=mx(\frac{a}{m}+\frac{y}{x}).$$
Set $d=\deg(\E)$ and $n= {\rm rank} (\E)$.  By  Lemma \ref{cubic-kahler} (6), the ratio of $u=x\xi+yF$ is given by $$\rho_{\pi}(u)=d +n \frac{y}{x}.$$
Therefore, for an $m-$section $Z$ with $[Z]=al+m\eta$,
$$\langle u, [Z]\rangle=mx(\frac{a}{m}+\frac{y}{x})\, >0$$
if $\frac{a}{m}\geq \alpha$ and 
$\rho_{\pi}(u)>d-n \alpha$.

The last statement is clear by taking $\alpha=d/n$. 
\end{proof}

\begin{prop} \label {stablecubic}
The K\"ahler cone
of $\mathbb P({\mathcal V})$ is the forward cone if  ${\mathcal V}$ is a semi-stable rank $n$ bundle. 

In the rank $2$ case, the converse is also true. 
\end{prop}

\begin{proof}
Since $\V$ is semi-stable, so are all symmetric powers $s^m(\V)$. Hence, by Lemma \ref{degree},  for any quotient line bundle $\mathcal M$ of $s^m\V$, we have 
$$\deg(\mathcal M)\geq \mu(s^m\V)= m\,\mu(\V)=m\frac{\deg (\V)}{n}.$$ 
Therefore if we set $d=\deg(\V)$ and apply it to $\mathcal M_Z$ with  $[Z]=al+m\eta$, then
$$a\, \geq \, \frac{md}{n}.$$
The first statement now follows from Lemma \ref{kah=pos}. 

In the rank $2$ case, if $\V$ is not semi-stable, then there is a quotient  line bundle with degree strictly less than $\frac{\deg (\V)}{2}$. Since degree is an integer, 
for the corresponding section $Z$ and $u=x\xi+yF$, we have 
$$\langle u, [Z]\rangle=\langle x\xi+yF, al+\eta\rangle=ax+y\, \leq  \frac{\deg (\V)-1}{2}x+y.$$
If we choose $u=4\xi+(1-2d)F$, then $u$ satisfies $u^2>0$ but pairs negatively with $[Z]$. 
\end{proof}

It is easy to see that   for a rank $r$ bundle 
$\V$, if the K\"ahler cone
of $\mathbb P({\mathcal V})$ is the forward cone  then there is no quotient  line bundle with degree strictly less than $\frac{\deg (\V)}{r}$.

\begin{lemma}\label{completedec}
Suppose $\V=\oplus_j \mathcal L_j $ with degrees $a_j$  and $a_1=\min a_i$.  If we view a line bundle  summand  also  as a quotient line bundle 
and let $C_i$ be the corresponding section, 
then the curve cone of $\mathbb P(\V)$ is bounded  by $[C_1]=a_1 l+ \eta$ and $l$.

The ratio of the K\"ahler cone is $\sum(a_j-a_1)$.
If  $a_1= 0$, then the ratio of the K\"ahler cone is  $\deg(\V)$. 
\end{lemma}

\begin{proof}
If $Z \rightarrow X$ is of degree $m$, then $Z$ corresponds to a quotient line bundle $\mathcal M_Z$ of $s^m\V$.
Observe that  $s^m\V$ is a direct sum of line bundles and the minimal degree of these line bundles is  $ma_1$.
Since line bundles are stable,  we get by  Lemma \ref{slope estimate}  
that $\deg \mathcal M_Z\ge ma_1$.
This means that   $\langle \xi, [Z]\rangle  \ge ma_1$ if $\langle  F, \xi\rangle=m$. 
On the other hand, $C_1$ is an effective curve with $\langle F, [C_1]\rangle =1$ and $\langle \xi, [C_1]\rangle =a_1$. Hence the two extremal rays are $[C_1]$ and $l$.

Recall $\eta$ is the class with $\langle \xi, \eta\rangle =0$ and $\langle F, \eta \rangle=1$. 
Then $[C_1]=a_1l+\eta$. 
By   the Kleiman criterion, $u=x\xi+yF$ is in the  K\"ahler cone of $\mathbb P(\V)$ if and only if 
\begin{equation}\label{x}
\langle x\xi +yF,  l\rangle=x>0, \quad   \langle x\xi+yF, [C_1]\rangle=a_1x+y>0.
\end{equation}
Write 
$$\rho(u)=\frac{x^{n-1}((\sum a_j) x+ny)}{x^n}= \sum(a_j-a_1)+n(a_1x+y)/x.$$
By \eqref{x}, the ratio of the K\"ahler cone is $\sum(a_j-a_1)$.
In particular,    when   $a_1=0$,  the K\"ahler cone  has ratio 
$ \deg (\V).$
\end{proof}

The above discussion  follows more clearly from Hartshorne's  theorem by looking at the twisting $\E\otimes \mathcal L_1^{-1}$ and notice $\mathbb P(\E)=\mathbb P(\E\otimes \mathcal L_1^{-1})$. Abusing the notation, we denote the new bundle to be $\E$ as well. By Theorem \ref{Hartshorne}, $\E$ is nef. When we take $C$ corresponding to the line bundle quotient $\mathcal O$, $\langle \xi, [C]\rangle=0$. Hence the boundary of curve cone is spanned by $[C]$ and $l$.

\subsubsection{The ratio of the restricted K\"ahler cone}

Let $\V$ be a rank $n$  quotient of $\E$.  
We introduce  the ratio of             the restricted K\"ahler cone  $\rho( \V, \E)$,
\begin{equation}\nonumber
\rho (\V, \E)=\inf \{\rho(u|_{\mathbb P(\V)})| u \hbox{ is the  class of a K\"ahler  form on $\mathbb P(\E)$}\}.
\end{equation} 
Recall that $\mathbb P_s(\V)=\mathbb P(\V^*)$ so the normal type   of $\mathbb P(\V)$  is actually $\deg(\V)$ rather than $-\deg(\V)$. 
We have seen  that the K\"ahler cone of a semi-stable   bundle is maximal. For the direct sum of a semi-stable  bundle and a line bundle we have

\begin{lemma}\label{sum}  Let $\mathcal E=\mathcal V\oplus \mathcal L$, where $\mathcal V$ is a semi-stable  bundle with
$\mu(\mathcal L)\geq  \mu(\mathcal V)$. 
Then the restricted K\"ahler cone of $\mathbb P(\V)$ is the forward cone. 
\end{lemma}
\begin{proof}  
 Let $Z$ be a multisection with $\langle F, [Z]\rangle=m$.  Then $Z$ corresponds to a quotient line bundle $\mathcal M_Z$ of $s^m(\E)$.  In our case, 
 \begin{equation} \label{symsum} s^m(\mathcal E)=s^m(\mathcal V\oplus \mathcal L)=s^m(\mathcal V)\oplus s^{m-1}(\mathcal V)\mathcal L\oplus \cdots\oplus s^{m-i}(\mathcal V)\mathcal L^i\oplus \cdots \oplus \mathcal L^m. 
\end{equation}
Since the symmetric powers of a semi-stable bundle over a curve are semi-stable, each summand in \eqref{symsum} is semi-stable. 
By Lemma \ref{degree},  we have 
$$\mu(s^{m-i}(\mathcal V)\mathcal L^i )=\mu(s^{m-i}(\mathcal V))+\mu(\mathcal L^i ) =(m-i)\mu(\mathcal V)+i\mu( \mathcal L)\geq m\mu(\V)$$
since $\mu(\mathcal L)\geq  \mu(\mathcal V)$. 
So  the minimal slope of the summands in \eqref{symsum} is achieved by $s^m(\V)$. Since each summand is semi-stable, by Lemma \ref{slope estimate}  the quotient line bundle $\mathcal M_Z$ of $s^m(\E)$ satisfies 
 $$\deg \mathcal M_Z\geq \mu (s^m \V)= m\mu(\V)=m\frac{\deg (\V)}{{\rm rank} (\V)}.$$
 By Lemma \ref{kah=pos}, $u=x\xi_{\E}+yF_{\E}$ is in the K\"ahler cone of $\mathbb P(\E)$ if 
 $y/x> -\deg(\V)/ {\rm rank} (\V)$.
 Since $u|_{\mathbb P(\V)}=x\xi_{\V}+yF_{\V}$,  it follows that the restricted K\"ahler cone of $\mathbb P(\V)$ is  the forward cone. 
\end{proof}

\begin{lemma}\label{normal p<0 blow-down}  
 Given the normal type $qn+t<0$ with $0\leq t\leq n-1$,  
 there exists $\V$ with  deg $\V=qn+t$ such that
 \begin{enumerate}[label=(\roman*)]
\item  
 $   \rho (\V, \V\oplus \mathcal O) =0$ if $g(\Sigma)>0$;
 
 \item    $\rho (\V, \V\oplus \mathcal O)=t $  if $g(\Sigma)=0$.
 \end{enumerate}
\end{lemma}

\begin{proof}
Suppose  $g(\Sigma)>0$.
Then there are semi-stable bundles of arbitrary rank $r\ge 2$ and degree $d$ over $\Sigma$ (\cite{Se} for $g\geq 2$ and  \cite{A} for $g=1$). 
Just pick $\V$ to be a semi-stable bundle with degree $qn+t\leq 0$ and apply Lemma \ref{sum}.

Now we assume $g(\Sigma)=0$ and apply Lemma \ref{completedec}. Consider the bundle $\V=\mathcal O(q)^{n-t}\oplus \mathcal O(q+1)^t$. 
By Lemma  \ref{completedec}, the extremal ray of $\PP(\V)$ is given by $\mathcal O(q)$ and $\rho(\V)=t$. 
Note that  $q\leq -1$, so the extremal ray of $\PP(\V\oplus \mathcal O)$ is also given by $\mathcal O(q)$. 
Hence the restricted K\"ahler cone is just the  K\"ahler cone. 
\end{proof}


\begin{lemma} \label{normal p>0 blow-down}
Given  the normal type  $(p-1)n+t\geq 0$ with $0\leq t\leq n-1$,   there is a completely decomposable $\V$ such that  
$\rho(\V)=t$ and $$
\rho(\V, \V\oplus \mathcal O)=\deg(\V)=(p-1)n+t.$$ 
\end{lemma}

\begin{proof} Consider a decomposable rank $n$ bundle 
$\V=\oplus_{i=1}^n \mathcal L_i\otimes \mathcal R$, where $\mathcal R$ has degree $(p-1)$, 
$\mathcal L_i=\mathcal O$ or has degree $1$. 
The number of $\mathcal O$ factors  is $n-t$. 
By Lemma \ref{completedec}, $\rho(\V)=t$.

For the completely decomposable bundle $\E=\V\oplus \mathcal O$,   the section  $s_{\mathcal O}$ is  extremal. Note this is just the
$\eta$ class of $\E$. 
Since the degree of $\mathcal L_i\otimes \mathcal R$ is $\geq 0$, by Lemma \ref{completedec}, the ratio $\rho(\V\oplus \mathcal O)$ is $\deg(\V\oplus \mathcal O)=\deg(\V)$. 
Since the $\eta$ class of $\V$ is the restriction of the $\eta$ class of $\V\oplus \mathcal O$, we apply Lemma \ref{cubic} to get the ratio $\rho(\V, \V\oplus \mathcal O)$ to be $\deg(\V)$.  
\end{proof}


\subsection{Almost standard symplectic forms}

Suppose $\pi:D\to \Sigma$ is a linear $\mathbb P^{n-1}$ bundle over a surface $\Sigma$.
Recall that  a fibred symplectic form on  $D$  is called standard if it arises from the Sternberg-Weinstein universal construction as described in Section 2.1.
In particular, a standard form restrict to a multiple of the Fubini-Study form on each fiber. 
And a  fibred symplectic form on $D$  is said to be almost standard  if it is 
 deformation 
 to a standard form via fibred forms. 
The  linear $\mathbb P^{n-1}$ bundle  $\pi:D\to \Sigma$ always has a holomorphic realization of the form $\mathbb P(\E)$. 
A K\"ahler form on $D$ refers to a K\"ahler form on any  holomorphic realization of this sort.

\begin{lemma} \label{almost standard}
Suppose $\pi:D\to \Sigma$ is a linear $\mathbb P^{n-1}$ bundle over a surface $\Sigma$.
Then the  space of almost standard forms on $D$ is path connected
and contains the   K\"ahler forms. 
\end{lemma}

\begin{proof} 
Since the base $\Sigma$ has dimension 2, by Proposition 4.4  in \cite{McDuff96 def-iso}, the space of standard forms which restricts to the same multiple of
the Fubini-Study form is path-connected. Therefore the space of all standard forms is also path connected by scaling. 
It follows that  the  space of almost standard forms is path connected since any almost standard form is connected to a standard form via
a path of almost standard forms. 

Fix a holomorphic realization $\mathbb P(\E)$ of $D$. Since the fibers of $\mathbb P(\E)$ are holomorphic, the K\"ahler forms on $\mathbb P(\E)$ are fibred. 
Moreover, there exists  standard K\"ahler forms on $\mathbb P(\E)$
({\it c.f.} Proposition 3.18 in \cite{Voisin book}).
Since the space of K\"ahler forms on $\mathbb P(\E)$ is path connected, any K\"ahler form on $\mathbb P(\E)$ is almost standard. 
\end{proof}

\begin{remark}
There is a unique deformation class of K\"ahler structures for each topological type 
since
the moduli space of holomorphic bundles over a curve with the fixed rank and degree is connected (\cite{Mumford}). 
This implies that the subspace of K\"ahler forms is also path connected. 
\end{remark}

\begin{lemma} \label{symplectic GW}
Suppose $D$ is a linear $\mathbb P^{n-1}$ bundle over $\mathbb P^1$. 
Then $\rho_{\pi}([\omega])>t_n(D)$  for an   almost standard  form $\omega$ on $D$.   
\end{lemma}

\begin{proof} 
Suppose $t_n(D)=n-k$. Model $D$  on $$V= \mathbb C^{k} \oplus \mathbb C(-1)^{n-k}, 1\leq k\leq n.$$ 
Let $S$ be the section $\mathbb P_s(\mathbb C)$.
We will show that  a  GW invariant of the curve class $\eta=[S]$  is nonzero. 
Then the inequality $\rho_{\pi}([\omega])>n-k= t_n(D)$ is a consequence of  Lemma \ref{cubic}.

Tensoring $V^*$ by $\mathbb C(1)$, we have $V^*\otimes \mathbb C(1)= \mathbb C(1)^{k} \oplus \mathbb C(2)^{n-k}$. 
Now $S$ corresponds to $S^*$, a quotient line bundle $\mathbb C(1)$ of $V^*\otimes \mathbb C(1)$. 
If $\omega$ is K\"ahler,  the GW invariant of $[S]$ is nonzero by   Lemma \ref{GW'}. 
By Lemma \ref{almost standard}  the space of almost standard forms on $D$ is path connected and contains K\"ahler forms, so $\omega$ is  deformation to a K\"ahler form  and has the same GW invariant. 
\end{proof}

\begin{definition}
Suppose $D$ is a linear $\mathbb P^{n-1}-$bundle over $\Sigma$.   
We define   
  the ratio $\rho_{\pi}(D)$ of the (almost standard) symplectic cone by 
  $$\rho_{\pi}(D)=\inf \{\rho_{\pi}(u)| \hbox{$u$ an almost standard symplectic class} \}.$$
If $D$ is a codimension $2$ submanifold of $M$ and $S\subset M$ is a submanifold disjoint from $D$,  we define the relative ratio 
$$\rho_{\pi} (D; M, S)$$ to be the infimum of $ \rho_{\pi}(u|_D)$ for $u$  a  class of a  symplectic form on $M$ that is  almost standard on $D$ and symplectic on $S$. 
\end{definition}

Clearly,  $\rho_{\pi}(D; M, S)\geq \rho_{\pi} (D)\geq 0.$


\begin{prop} \label{lower bound}
Suppose $D$ is  a linear $\mathbb P^{n-1}-$bundle  over $\Sigma$. Then 
\begin{equation} \label { } 
    \rho_{\pi} (D) = \left\{ \begin{array}{ll}0  & \hbox{if $g(\Sigma)>0$}, \\
t_n(D)       &\hbox{if $g(\Sigma)=0$}.\\
\end{array}
\right.
\end{equation} 
Moreover, any ratio can be realized by a K\"ahler form. 
\end{prop}

\begin{proof}
When $g(\Sigma)>0$, 
by Proposition  \ref{stablecubic}, for  any  complex structure arising from a semi-stable  bundle, 
the ratio of the K\"ahler cone is $0$.

When the base is $\mathbb P^1$,     $\rho_{\pi}(D)\geq     t_n(D)$ by  Lemma \ref{symplectic GW}.
The reverse inequality relies on the K\"ahler cone computation in  Lemma  \ref{completedec}.
 $D$  can be modeled on a holomorphic bundle $\V$ as in Lemma  \ref{completedec} with $a_1=0$. 
 The ratio of the K\"ahler cone is then $\deg (\V)$  by Lemma  \ref{completedec}. So $\rho_{\pi}(D)\leq     \deg(\V)$. 
Notice that $\mathbb P(\V)=\mathbb P_s(\V^*)$ and  $\deg(\V)=-\deg(\V^*)$.
Recall the topological type for $\mathbb P_s(E)$ is  the smallest non-negative integer  congruent to $-\deg(E)$ modulo $n$. 
Therefore $\rho_{\pi}(D)=t_n(D)$ when the base is $\mathbb P^1$. 
\end{proof}

We are ready to confirm Proposition \ref{almost standard cone}. 

\begin{proof} [Proof of Proposition \ref{almost standard cone}]
Claim  (i)  follows directly from Lemma \ref{almost standard} and Proposition \ref{lower bound}. For Claim (ii) on the symplectic cone when $n$ is odd and $g(\Sigma)>0$, it follows from Proposition \ref{lower bound} once we  note that $-\omega$ is also a symplectic
form compatible with the orientation if $\omega$ is. 
\end{proof}

\begin{prop} \label{relative ratio}  Suppose $V$ has rank $n$ and degree $d$. Let  $K=\mathbb P_s(V\oplus \mathbb C)$, $D=\mathbb P_s(V)$ and $S=\mathbb P_s(\mathbb C)$.  Then 
\begin{equation} \label{rel ratio}
    \rho_{\pi}(D; K, S)=    \left\{ \begin{array}{ll}\max\{0, -\deg(V)\}  & \hbox{if $g(\Sigma)>0$}, \\
\max\{t_n(-\deg V),-\deg(V)\}      &\hbox{if $g(\Sigma)=0$}.\\
\end{array}
\right.
\end{equation} 
Moreover, the  symplectic form on $K$ could be chosen to be  $S^1-$invariant with respect to the natural semi-free $S^1-$action from the splitting $V\oplus \mathbb C$.
\end{prop}

\begin{proof} 
Let $\omega$ be  a  symplectic form on $K$ that is   almost standard on $D$ and symplectic on $S$. 
Since the  $\eta$ class $\eta_D$ of the linear $\mathbb P^{n-1}$ bundle  $D=\mathbb P_s(V)$ is sent to the $\eta$ class $\eta_K$  of the 
$\mathbb P^{n}$ bundle $K=\mathbb P_s(V\oplus \mathbb C)$ under the inclusion $D\to K$ and $\eta_K$   is represented by the symplectic section $S$,   $\omega|_D$ is positive on $\eta_D$. 

Therefore we have the inequality  $$\rho_{\pi}(D; K, S)\geq  -\deg(V)$$ by Lemma \ref{cubic}. 
The equality \eqref{rel ratio} then follows from the K\"ahler constructions on $\mathbb P(\mathcal V\oplus \mathcal O)$ in Lemmas \ref{normal p<0 blow-down} and \ref{normal p>0 blow-down}.

It remains to  verify the last statement. 
The holomorphic bundle $\mathbb P(\mathcal V\oplus \mathcal O)$ has a natural semi-free holomorphic $S^1-$action induced from the splitting $\mathcal V\oplus \mathcal O$ that fixes both $D$ and $S$.  
For each element $g$ of this $S^1$ automorphism of $\mathbb P(\mathcal V\oplus \mathcal O)$, 
$g^*\Omega$ is still  a K\"ahler form in the class $[\Omega]$. 
As the space of K\"ahler forms is convex, take average of $g^*\Omega$ with respect to  the $S^1-$action, we have a K\"ahler form $\Omega'$ in class $[\Omega]$ which is invariant under this $S^1-$action. 
\end{proof}

\section{Proof of theorems}
Finally, we prove Theorem \ref{main}, Theorem \ref{coh matching} and Theorem \ref{ruled}.

\begin{proof} [Proof of Theorem \ref{main}]
We have a symplectic divisor $\pi:D^{2n}\to \Sigma$ of $ (M^{2n+2}, \omega)$ arising from  symplectically    blowing up a symplectic surface $\Sigma$ 
in a symplectic manifold.
Observe that from the symplectic cut description of symplectic blowing up, $(M, \omega)=(\overline X^-, \omega^-)$ and $D=Z^-$.
On the $\overline X^+$ side,  we have a projective bundle triple by Lemma \ref{triple},  the $\mathbb P^{k}-$bundle $\overline X^+=\mathbb P_s(N_{\Sigma}\oplus {\mathbb C})$, 
 the $\mathbb P^{k-1}-$bundle $Z^+=\mathbb P_s(N_{\Sigma}\oplus 0)$ and  a copy of  $\Sigma$ which is the infinity  section ${\mathbb P}_s(0\oplus {\mathbb C})$.
Moreover, $c_1(N_D)=-c_1(N_{Z^+})$. 
By Lemma \ref{Xi}(3) snd Lemma \ref{cohomology ring}, $ \int_D c_1(N_{Z^+})^n=-\deg(N_\Sigma)$.
Therefore 
$c_1(N_D)$ satisfies
$$ \int_D (-1)^n c_1(N_D)^n=-\deg(N_\Sigma).$$

It  remains to prove the ratio inequality
\begin{equation}
\label {constraint'}  
    \rho_{\pi} ([\omega|_D]) >
     \left\{ \begin{array}{ll} -\deg(N_{\Sigma}), & \hbox{if $g(\Sigma)>0$}, \\
\max\{t_n(-\deg N_{\Sigma}),  -\deg(N_{\Sigma})\},       &\hbox{if $g(\Sigma)=0$}.\\
\end{array}
\right.
\end{equation} 
Again we  can resort to the $\overline X^+$ side since 
\begin{equation}\label{D and Z+}
 \rho_{\pi}([\omega|_D])=\rho_{\pi^-}([\omega^-|_{Z^-}])=\rho_{\pi^+}([\omega^+|_{Z^+}]),
\end{equation}
where $\pi^{\pm}:Z^{\pm}\to Y$ is the projection. Note  that $[\omega^+]$ pairs positively with the $\eta$  class of the $\mathbb P^{k}-$bundle $X^+$, which is represented by the symplectic section ${\mathbb P}_s(0\oplus {\mathbb C})$. 
Since the $\eta$ class of the $\mathbb P^k$-bundle $X^+$ is also the $\eta$ class of the $\mathbb P^{k-1}-$subbundle $Z^+$, $[\omega^+|_D]$ pairs positively with the $\eta$ class of the $\mathbb P^{k-1}-$bundle $Z^+$.

Note that the $\mathbb P^{k-1}-$bundles   $Z^+$ is  modeled on $N_{\Sigma}$. By \eqref{D and Z+}, the ratio inequality \eqref{constraint'} 
 follows from Lemma \ref{cubic} when $g(\Sigma)>0$ and 
 Lemmas \ref{cubic} and \ref{symplectic GW} when $g(\Sigma)=0$. 
\end{proof}

\begin{proof}[Proof of Theorem \ref{coh matching}]
Here we have 
 a topological exceptional divisor  $\pi: (D, \omega|_D)\to \Sigma$ of $(M, \omega)$
 satisfying the ratio bound  
\begin{equation} \label {ratio constraint} 
    \rho_{\pi} ([\omega|_D]) >
     \left\{ \begin{array}{ll}\alpha_{D, M},  & \hbox{if $g(\Sigma)>0$}, \\
\max\{\alpha_{D, M}, \quad t_n(\alpha_{D, M})\},      &\hbox{if $g(\Sigma)=0$}.\\
\end{array}
\right.
\end{equation} 
where $\alpha_{D, M}=\int_D (-1)^n c_1(N_D)^n$.

 We model the linear $\mathbb P^{k-1}-$bundle $D$  by a complex rank $k$ vector bundle $V$ over $\Sigma$ with $\deg V=-\alpha_{D, M}$. 
 Let  $$K=\mathbb P_s(V\oplus \mathbb C), \quad D'=\mathbb P_s(V), \quad S=\mathbb P_s(\mathbb C)$$ as in Proposition \ref{relative ratio}.
  Since $[\omega|_D]$ satisfies the ratio inequality
 \eqref{ratio constraint}, 
there exists an $S^1-$equivariant symplectic form $\Omega$ on $K$ such that $[\Omega|_{D'}]=[\omega|_D]$
by Proposition \ref{relative ratio}. 
Clearly, the triple $(K, D', S; \Omega)$ is a desired weak matching triple. 
\end{proof}



Since  Theorem \ref{ruled} has several statements we restate   it here and deduce it from Theorem \ref{main} and Proposition \ref{dim 6}. 
 
\begin{theorem}  Let $(M, \omega)$ be a $6-$dimensional symplectic manifold and $D$ a codimension $2$ symplectic submanifold.
 Suppose $D$ admits a  $\mathbb P^1-$bundle structure $\pi:D\to \Sigma$   over a surface $\Sigma$ with $\langle c_1(N_D), l\rangle =-1$. 
Let $$\alpha_{D, M}=c_1(N_D)\cdot c_1(N_D)\quad \hbox{ and} \quad  \rho=\rho_{\pi}([\omega|_D]).$$
Suppose  $(D, \omega|_D)$ arises from   blowing up a surface.  Then $\rho \ne 2$ if   $D= S^2\times S^2$ with $\alpha_{D, M}=2$,  
and $\rho>\alpha_{D, M}$  otherwise.

Conversely,  if  $\rho>\alpha_{D, M}$,  
$(M, \omega)$ can be 
blown down  along $D$ up to deformation.   In particular, this is  the case if $\alpha_{D, M}\leq  0$.
Moreover, when  $D=S^2\times S^2$ with $\alpha_{D, M}=2$, $(M, \omega)$ can be  
blown down  along $D$ up to deformation as long as $\rho\ne 2$.

 \end{theorem}

\begin{proof} 
First of all, $(D, \omega|_D)$ is a topological exceptional divisor since any $\mathbb P^1-$bundle over $\Sigma$  is linear and any symplectic form is fibred and almost standard.  

We first assume $g(\Sigma)>0$. 
If $(D, \pi, \omega|_D)$ is a (symplectic) exceptional divisor, then  $\rho>\alpha_{D, M}$ by Theorem \ref{main}. 
Conversely, 
 $(D, \omega)$ is admissible if $\rho>\alpha_{D, M}$. 
By Proposition \ref{dim 6}, up to integral deformation,  $D$ can be symplectically blow down.

Suppose $g(\Sigma)=0$. There are two cases, $\alpha_{D,M}$ is even or $\alpha_{D, M}$ is odd. 
When $\alpha_{D, M}$ is odd,  
$D$ is the unique non-trivial $\mathbb P^1-$bundle over $\mathbb P^1$, which 
is also the 1 point blowup of $\mathbb P^2$. 
In this case, $t_n(\alpha_{D, M})=1$ and 
$\rho>1$ for any symplectic form on $D$. 
So, as in the case $g(\Sigma)>0$,  
$\rho>\alpha_{D, M}$ if $(D, \pi, \omega|_D)$ is a (symplectic) exceptional divisor, 
and conversely, 
 $(D, \omega)$ is admissible  and hence can be symplectically blow down up to integral deformation if $\rho>\alpha_{D, M}$.

When $\alpha_{D, M}$ is even,  $D=S^2\times S^2$ and $t_n(\alpha_{D, M})=0$. 
There are two $\mathbb P^1$ bundle structures on $D$. 
Accordingly we further divide into  two cases, $\alpha_{D, M}\ne 2$ and $\alpha_{D, M}=2$. 
When $\alpha_{D, M}\ne 2$, there is only one fibration with the appropriate normal bundle condition and
the admissible condition is again  $\rho>\alpha_{D, M}$. So the conclusion is exactly  the same as the $g(\Sigma)>0$ case.

When   $\alpha_{D, M}=2$,  the evaluation of $c_1(N_D)$  is $-1$ on  both  fiber classes. 
Suppose the symplectic areas of the two rulings are $x$ and $y$. 
Then the symplectic volume is $2xy$. 
The ratio of $\omega|_D$ with respect to the first ruling is $2xy/x^2=2y/x$, which is bigger than $\alpha_{D, M}=2$ if and only if  $y>x$. Similarly, 
the ratio of $\omega|_D$ with respect to the second  ruling  is bigger than $\alpha_{D, M}=2$ if and only if  $x>y$.
Therefore $\rho\ne 2$ if $D$ is a symplectic exceptional divisor. Conversely, if $\rho\ne 2$, up to integral deformation
we could blow down $(M, \omega)$ along $D$ with respect to the fibers with smaller area. 
\end{proof}

\end{document}